\documentclass[preprint]{imsart}

\usepackage{graphicx,amsmath,amssymb,epsfig}

\startlocaldefs

\long\def\comment#1{}
\newtheorem{theorem}{Theorem}
\newtheorem{corollary}{Corollary}
\newtheorem{lemma}{Lemma}

\newtheorem{remark}{Comment}[section]
\newenvironment{proof}[1][Proof]{\textbf{#1.} }{\ \rule{0.5em}{0.5em}}
\newcommand{\qed}{\rule{0.5em}{0.5em}}
\newcommand{\wpe}{wp $1-\varepsilon$}

\def\RR{{\rm I\kern-0.18em R}}
\def\vol{{\rm vol}}

\def\conv{{\rm  conv}}

\def\est{\hat{\mu}_{g}}
\def\Aproof#1#2{ {\noindent \bf Proof of #1 \ref{#2}.} }

\newcommand{\be}{\begin{eqnarray}}
\newcommand{\ee}{\end{eqnarray}}

\newcommand{\NN}{\mathbf{N}}

\newcommand{\ba}{\begin{array}}
\newcommand{\ea}{\end{array}}
\newcommand{\bs}{\begin{align}\begin{split}\nonumber}
\newcommand{\bsnumber}{\begin{align}\begin{split}}
\newcommand{\es}{\end{split}\end{align}}

\renewcommand{\(}{\left(}
\renewcommand{\)}{\right)}
\renewcommand{\[}{\left[}
\renewcommand{\]}{\right]}
\renewcommand{\hat}{\widehat}

\newcommand{\Ep}{E}

\endlocaldefs

\begin{document}

\begin{frontmatter}

\title{On the Computational Complexity of MCMC-based
Estimators in Large Samples}\runtitle{Complexity of MCMC}


\author{\fnms{Alexandre}
\snm{Belloni}\ead[label=e1]{abn5@duke.edu}}\footnote{Research
support from a National Science Foundation grant is gratefully
acknowledged. IBM Herman Goldstein Fellowship is also gratefully
acknowledged.}
\affiliation{Duke University}
\and
\author{\fnms{Victor} \snm{Chernozhukov}\ead[label=e2]{vchern@mit.edu}}\footnote{Research support from a National Science Foundation
grant is gratefully acknowledged. Sloan Foundation Research
Fellowship and Castle Krob Chair are also gratefully
acknowledged. \\ First version: April 2006.}
\affiliation{Massachusetts Institute of Technology}
\runauthor{Belloni and Chernozhukov}

\begin{abstract}
In this paper we examine the implications of the
statistical large sample theory for the
computational complexity of Bayesian and
quasi-Bayesian estimation carried out using
Metropolis random walks.  Our analysis is
motivated by the Laplace-Bernstein-Von Mises
central limit theorem, which states that in large
samples the posterior or quasi-posterior
approaches a normal density.  Using the conditions
required for the central limit theorem to hold, we establish polynomial bounds on
the computational complexity of general
Metropolis random walks methods in large samples.
Our analysis covers cases where the underlying
log-likelihood or extremum criterion function is
possibly non-concave, discontinuous, and with
 increasing parameter dimension. However, the central limit
theorem restricts the deviations from continuity
and log-concavity of the log-likelihood or
extremum criterion function in a very specific
manner.

Under minimal assumptions required for the central limit
theorem to hold under the increasing parameter dimension, we show that the
Metropolis algorithm is theoretically efficient
even for the canonical Gaussian walk which is
studied in detail. Specifically, we show that the
running time of the algorithm in large samples is
bounded in probability by a polynomial in the
parameter dimension $d$, and, in particular, is
of stochastic order $d^2$ in the leading cases
after the burn-in period. We then give
applications to exponential families, curved exponential
families, and Z-estimation of increasing dimension.

\end{abstract}





\end{frontmatter}

\section{Introduction}

Markov Chain Monte Carlo (MCMC) algorithms have dramatically
increased the use of Bayesian and quasi-Bayesian methods for
practical estimation and inference. (See e.g. books of Casella and
Robert \cite{RC}, Chib \cite{Chib-Chapter}, Geweke
\cite{Gew-Chap}, Liu \cite{Liu-Book} for detailed treatments of
the MCMC methods and their applications in various areas of
statistics, econometrics, and biometrics.)  Bayesian methods rely
on a likelihood formulation, while quasi-Bayesian methods replace
the likelihood with other criterion functions. This paper studies the
computational complexity of MCMC algorithms (based on Metropolis
random walks) as both the sample and parameter dimensions grow to
infinity at the appropriate rates. The paper shows how and when the
large sample asymptotics places sufficient restrictions on the
likelihood and criterion functions that guarantee the efficient --
that is, polynomial time -- computational complexity of these
algorithms. These results suggest that at least in large samples,
Bayesian and quasi-Bayesian estimators can be computationally
efficient alternatives to maximum likelihood and extremum
estimators, most of all in cases where likelihoods and criterion
functions are non-concave and possibly non-smooth in the parameters of
interest.

To motivate our analysis, let us consider the Z-estimation problem, which is a basic method for estimating various kinds of structural models, especially in biometrics and econometrics.  The idea behind this approach is to maximize some criterion function: \bsnumber\label{Qn} Q_n\(\theta\) = -
\left\| \frac{1}{\sqrt{n}} \sum_{i=1}^n  m (U_i,\theta) \right\|^2, \quad \theta \in \Theta
\subset \RR^d,
\end{split}\end{align}
where $U_i$ is a vector of random variables, and $m(U_i, \theta) $ is a vector of functions such that $E[m(U_i, \theta)] = 0$ at the true parameter value $\theta = \theta_0$.
 For example, in estimation of conditional $\alpha$-quantile models with censoring and endogeneity, the functions take the form
\bsnumber\label{quantile example}
m(U_i, \theta) = W ( \alpha/p_i(\theta) -  1( Y_i \leq X_i\theta))Z_i.
\end{split}\end{align}
Here $U_i=(Y_i,X_i,Z_i)$, $Y_i$ is the response variable, $X_i$ is a vector of
regressors; in the censored regression models, $Z_i$ is the same as $X_i$, and
$p_i(\theta)$ is a weighting function that depends on the probability of censoring that depends on $X_i$ and $\theta$ (see \cite{LTW} for extensive motivation and details), and in the endogenous models, $Z_i$ is  a vector of instrumental variables that affect the outcome variable $Y_i$ only through $X_i$ (see \cite{CH} for motivation and details), while $p_i(\theta)=1$ for each $i$;  the matrix $W$ is some positive definite weighting matrix. Finally, the index $\alpha \in (0,1)$ is the quantile index, and $X_i'\theta$ is the model for the $\alpha$-th quantile function of the outcome $Y_i$.

In these quantile examples, the criterion function $Q_n(\theta)$ is highly discontinuous and non-concave, implying that the argmax estimator may be difficult or impossible to obtain.  Figure \ref{Fig:1} in Section \ref{Sec:Setup} illustrates this example and similar examples where the argmax computation is intractable, at least when the parameter dimension $d$ is high.  In typical applications, the parameter dimension $d$ is indeed high in relation to the sample size (see e.g. Koenker \cite{K1988} for a relevant survey).  Similar issues can also arise in M-estimation problems, where the extremum criterion function takes the form, $Q_n\(\theta\) = \sum_{i=1}^n m(U_i,\theta)$, where $U_i$ is  a vector of random variables, and $m(U_i,\theta)$ is a real-valued function, for example, the log-likelihood function of $U_i$ or some other pseudo-log-likelihood function. Section \ref{Sec:Appl} discusses several examples of this kind.

As an alternative to argmax estimation in both the Z- and M-estimation frameworks, consider the quasi-Bayesian estimator obtained by integration in place of optimization:
\begin{equation}\label{qBest} \displaystyle \hat \theta =
\frac{\displaystyle \int_{\Theta} \theta \exp\{Q_n(\theta)\}
d\theta}{\displaystyle \int_{\Theta} \exp\{Q_n(\theta')\} d
\theta'}.
\end{equation}
This estimator may be recognized as a quasi-posterior mean of the quasi-posterior density $\pi_n(\theta) \propto \exp{Q_n(\theta)}$. (Of course, when $Q_n$ is a log-likelihood, the term ``quasi" becomes redundant.) This estimator is not affected by local discontinuities and non-concavities and is often much easier to compute in practice than the argmax estimator, particulary in the high-dimensional setting; see, for example, the discussion in Liu, Tian, and Wei \cite{LTW} and Chernozhukov and Hong \cite{CH}.

At this point, it is worth emphasizing that we will formally
capture the high parameter dimension by using the framework of
Huber \cite{Huber}, Portnoy \cite{Portnoy}, and others.  In this
framework, we have a sequence of models (rather than a fixed
model) where the parameter dimension grows as the sample size
grows, namely, $d \to \infty $ as $n \to \infty$, and we will
carry out all of our analysis in this framework.

This paper will show that if the sample size $n$ grows to infinity and the dimension of the problem $d$ does not grow too quickly relative to the sample size, the quasi-posterior
 \bsnumber\label{QB}
\frac{\displaystyle \exp\{Q_n(\theta)\}}{\displaystyle
\int_{\Theta} \exp\{Q_n(\theta')\} d \theta'}
 \end{split}\end{align}
will be approximately normal. This result in turn leads to the main claim: the estimator (\ref{qBest}) can be computed using Markov Chain Monte Carlo in \textit{polynomial time}, provided that the starting point is drawn from the approximate support of the quasi-posterior
(\ref{QB}). As is standard in the literature, we measure running time in the number of evaluations of the numerator of the quasi-posterior function (\ref{QB}) since this accounts for most of the computational burden.

In other words, when the central limit theorem (CLT) for the quasi-posterior holds, the estimator (\ref{qBest}) is computationally tractable. The reason is that the CLT, in addition to implying the approximate normality and attractive estimation properties of the estimator $\hat \theta$, bounds non-concavities and discontinuities of $Q_n(\theta)$ in a specific manner that implies that the computational time is polynomial in the parameter dimension $d$. In particular, in the leading cases the bound on the running time of the algorithm after the so-called burn-in period is $O_p(d^2)$. Thus, our main insight is to bring the structure implied by the CLT into the computational complexity analysis of the MCMC algorithm for computation of (\ref{qBest}) and sampling from (\ref{QB}).

Our analysis of computational complexity builds on several
fundamental papers studying the computational complexity of
Metropolis procedures, especially Applegate and Kannan \cite{AK},
Frieze, Kannan and Polson \cite{FKP}, Polson \cite{P}, Kannan,
Lov\'asz and Simonovits \cite{KLS}, Kannan and Li \cite{KLi},
Lov\'asz and Simonovits \cite{LS}, and Lov\'asz and Vempala
\cite{LV01,LV02,LV03}. Many of our results and proofs rely upon
and extend the mathematical tools previously developed in these
works. We extend the complexity analysis of the previous
literature, which has focused on the case of an arbitrary concave
log-likelihood function, to the nonconcave and nonsmooth cases.  The
motivation is that, from a statistical point of view, in concave
settings it is typically easier to compute a maximum likelihood or
extremum estimate than a Bayesian or quasi-Bayesian estimate, so
the latter do not necessarily have practical appeal.  In contrast,
when the log-likelihood or quasi-likelihood is either nonsmooth, nonconcave, or both, Bayesian and quasi-Bayesian estimates defined
by integration are relatively attractive computationally, compared
to maximum likelihood or extremum estimators defined by
optimization.

Our analysis relies on statistical large sample
theory. We invoke limit theorems for posteriors
and quasi-posteriors for large samples as $n \to
\infty$. These theorems are necessary to support
our principal task -- the analysis of the
computational complexity under the restrictions
of the CLT.  As a preliminary step of our
computational analysis, we state a CLT for
quasi-posteriors and posteriors under parameters of increasing dimension, which extends the
CLT previously derived in the literature for
posteriors and quasi-posteriors for fixed
dimensions. In particular, Laplace c. 1809, Blackwell \cite{Blackwell}, Bickel
and Yahav \cite{BY}, Ibragimov and Hasminskii
\cite{IH}, and Bunke and Milhaud \cite{BM}
provided CLTs for posteriors.  Blackwell \cite{Blackwell}, Liu, Tian,
and Wei \cite{LTW}, and Chernozhukov and Hong
\cite{CH} provided CLTs for quasi-posteriors
formed using various non-likelihood criterion
functions. In contrast to these previous results,
we allow for increasing dimensions. Ghosal
\cite{G2000} previously derived a CLT for
posteriors with increasing dimension for log-concave
exponential families. We go beyond this canonical
setup and establish the CLT for the non-log-concave and
discontinuous cases. We also allow for general
criterion functions to replace likelihood
functions. This paper also illustrates the
plausibility of the approach using exponential families,
curved exponential families, and Z-estimation problems.  The curved
families arise for example when the data must
satisfy additional moment restrictions, as e.g.
in Hansen and Singleton \cite{hansen:singleton},
Chamberlain \cite{chamberlain}, and Imbens
\cite{imbens}. Both the curved exponential families and Z-estimation problems typically fall outside
the log-concave framework.

The rest of the paper is organized as follows.  In Section \ref{Sec:Setup}, we
establish a generalized version of the Central Limit Theorem for
Bayesian and quasi-Bayesian estimators.  This result may be seen
as a generalization of the classical Bernstein-Von-Mises theorem,
in that it allows the parameter dimension to grow as the sample
size grows.  In Section \ref{Sec:Setup},
we also formulate the main problem, which is to characterize the
complexity of MCMC sampling and integration as a function of the
key parameters that describe the deviations of the quasi-posterior
from the normal density. Section \ref{Sec:Sampling} explores the structure set
forth in Section \ref{Sec:Setup} to find bounds on conductance and mixing time
of the MCMC algorithm. Section \ref{Sec:Integration} derives bounds on the integration
time of the standard MCMC algorithm. Section \ref{Sec:Appl} considers an application to
a broad class of curved exponential families and Z-estimation problems, which  have possibly
non-concave and discontinuous criterion functions, and verifies that our results apply
to this class of statistical models.  Section \ref{Sec:Appl} also verifies that the high-level conditions of Section \ref{Sec:Setup} follow from the primitive conditions for these
models.

\begin{remark}[Notations.]
Throughout the paper, we follow the framework of high dimensional parameters introduced in Huber (1973). In this framework the parameter $\theta^{(n)}$ of the model, the parameter
space $\Theta^{(n)}$, its dimension $d^{(n)}$, and all other
properties of the model itself are indexed by the sample size $n$, and  $d^{(n)} \to \infty$ as $n \to \infty$. However, following Huber's convention, we will omit the index and
write, for example,  $\theta$, $\Theta$, and $d$  as abbreviations
for $\theta^{(n)}$, $\Theta^{(n)}$, and $d^{(n)}$, and so on.
\end{remark}

\section{The Setup and The Problem}\label{Sec:Setup}

Our analysis is motivated by the problems of estimation and
inference in large samples under high dimension. We consider a ``reduced-form" setup
formulated in terms of parameters that characterize local
deviations from the true parameter value.  The local
parameter $\lambda$ describes contiguous deviations from the true
parameter shifted by a first order approximation to an
extremum estimator $\tilde \theta$. That is, for $\theta$ denoting a parameter
vector, $\theta_0$ the true value, and $s= \sqrt{n} (\tilde \theta
- \theta_0)$ the normalized first order
approximation of the extremum estimator, we define the local parameter $\lambda$
as
$$
\lambda = \sqrt{n} (\theta - \theta_0) - s.
$$
The parameter space for $\theta$ is $\Theta$, and the parameter
space for $\lambda$ is therefore $\Lambda = \sqrt{n}(\Theta -
\theta_0) -s$.

The corresponding localized likelihood or localized criterion function is denoted by $\ell(\lambda)$. For example, suppose $L_n(\theta)$ is the
original likelihood function in the likelihood framework or, more
generally, $L_n(\theta)$ is $\exp\{Q_n(\theta)\}$ where
$Q_n(\theta)$ is the criterion function in extremum framework, then
$$
\ell(\lambda) = L_n(\theta_ 0 + (\lambda
+s)/\sqrt{n})/L_n(\theta_0 + s/\sqrt{n}).
$$
The assumptions below will be stated directly in terms of
$\ell(\lambda)$. In Section \ref{Sec:Appl}, we further illustrate the connection between the localized set-up and the non-localized set-ups and provide more primitive conditions within the exponential family, curved exponential family, and Z-estimation framework.

Then, the posterior or quasi-posterior density for  $\lambda$
takes the form, implicitly indexed by the sample size $n$,
\be\label{Def:f} f(\lambda) = \frac{\ell(\lambda)}{\int_{\Lambda}
\ell(\omega) d \omega}, \ee
 and we impose conditions that force the posterior to satisfy a CLT in the sense
 of approaching the normal density
\be\label{Def:phi}
\phi(\lambda) =  \frac{1}{ (2\pi)^{d/2} \det{(J^{-1})}^{1/2}} \exp
\(- \frac 12 \lambda' J \lambda\).
 \ee
More formally, the following conditions are assumed to hold for
$\ell(\lambda)$ as the sample size and parameter dimension grow to infinity:
$$
n \to \infty \ \text{ and } d \to \infty.
$$
We call these conditions the ``CLT
conditions":
 \begin{itemize}

\item[\textbf{C.1}]  The local parameter $\lambda$ belongs to the
local parameter space $\Lambda \subset \Bbb{R}^d$. The
vector $s$ is a zero mean vector with variance $\Omega$, whose
eigenvalues are bounded above as $n \to \infty$, and $\Lambda = K
\cup K^c$, where $K$ is a closed ball $B(0,\|K\|)$ such that $\int_{K} f(\lambda) d\lambda \geq 1-o_p(1)$
and $\int_{K} \phi(\lambda) d\lambda\geq 1-o(1)$.

\item[\textbf{C.2}] The lower semi-continuous posterior or
quasi-posterior function $\ell(\lambda)$ approaches a quadratic
form in logs, uniformly in $K$, i.e., there exist positive
approximation errors $\epsilon_1$ and $\epsilon_2$ such that for
every $\lambda \in K$, \be\label{Bounde1e2}
\left|\ln \ell(\lambda) - \left( - \frac{1}{2}\lambda' J \lambda\right) \right| \leq \epsilon_{1} + \epsilon_{2} \cdot \lambda' J \lambda/2,%
\ee %
\noindent where $J$ is a symmetric positive definite matrix with
eigenvalues bounded away from zero and from above uniformly in the
sample $n$. Also, we denote the ellipsoidal norm induced by $J$ as
$\|v\|_J := \| J^{1/2}v\|$.

\item[\textbf{C.3}] The approximation errors $\epsilon_1$ and
$\epsilon_2$ satisfy $\epsilon_{1}  = o_p(1)$, and $
\epsilon_{2}\cdot \|K\|^2_J = o_p(1)$.
\end{itemize}

\begin{remark}\label{comment: on conditions}
We choose the support set $K = B(0, \|K\|)$, which is a ball of radius $\|K\| =\sup_{\lambda \in K} \|\lambda\|$,  as follows. Under increasing dimension, the normal density is subject to a concentration of measure, namely that selecting $\|K \| \geq  C \cdot \sqrt{d}$, for a sufficiently large constant $C$, is enough to contain the support of the standard normal vector. Indeed, let $Z \sim N(0, I_d)$, then $Pr(Z \not \in K) = Pr (\|Z\|^2 > C^2 d) \to 0$ for $C>1$ as $d \to \infty$, because
$\|Z\|^2/d \to_p 1$.  For the case where  $W \sim N(0, J^{-1}) =
J^{-1/2} Z$, we have that $Pr(W \not \in K) \leq Pr (
\|Z\|/\sqrt{\lambda_{\min}} > \|K\|) \to 0$ for $\|K\| \geq C
\sqrt{d/\lambda_{\min}}$ for $C> 1$ as $d \to \infty$, where  $\lambda_{\min}$
denotes the smallest eigenvalue of $J$.  Moreover,  since $\|K\|_J = \lambda_{\max} \|K\| $,  where $\lambda_{\max}$
denotes the largest eigenvalue of $J$,
we need to have that $\|K\|_J > \sqrt{d\lambda_{\max}/\lambda_{\min}} $.  In view of condition C.3, this requires $\epsilon_2 d \lambda_{\max}/\lambda_{\min} = o_p(1)$
and hence $\epsilon_2 d = o_p(1)$.    Thus, in some of the computations presented below, we will set
$$
\| K \| = C \sqrt{d/\lambda_{\min}} \text{ and }  \|K\|_J= C \sqrt{d\lambda_{max}/\lambda_{min}} \text{ for } C>1.
$$
Finally, even though we make the assumption of bounded eigenvalues of $J$,  we will emphasize the dependence on the eigenvalues in most proofs and formal statements. This  will allow us to see immediately the impact of changing this assumption.
\end{remark}

These conditions imply that
$$\ell(\lambda) = a(\lambda)\cdot m(\lambda)$$
\noindent over the approximate support set $K$, where
\begin{equation}\label{Def:g}
 \ln a(\lambda) =   -\frac{1}{2}\lambda' J \lambda,
\end{equation}
\begin{equation}\label{Def:m}
- \epsilon_1 - \epsilon_2 \lambda'J\lambda/2 \leq \ln m(\lambda)
\leq \epsilon_1 + \epsilon_2 \lambda'J\lambda/2.
\end{equation}
Figure \ref{Fig:1} illustrates the kinds of deviations of $\ln
\ell(\lambda)$ from the quadratic curve captured by the parameters
$\epsilon_1$ and $\epsilon_2$, and also shows the types of
discontinuities and non-convexities permitted in our framework.
Parameter $\epsilon_1$ controls the size of local discontinuities
and parameter $\epsilon_2$ controls the global tilting away from
the quadratic shape of the normal log-density.

\begin{figure}
\begin{center}
\includegraphics[width=0.6\textwidth]{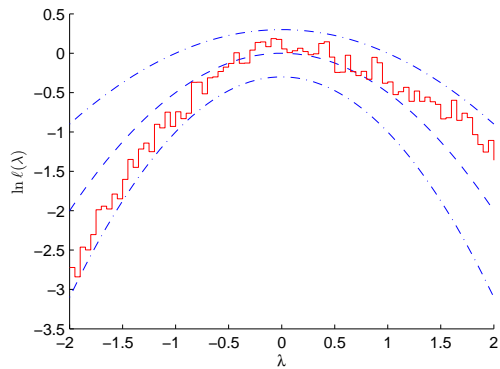}
\end{center}
\caption{This figure illustrates how $\ln \ell(\lambda)$ can
deviate from {\bf $ \ln a(\lambda)$}, allowing for possible
discontinuities in $\ln \ell(\lambda)$.}\label{Fig:1}
\end{figure}

\begin{theorem}[Generalized CLT for Quasi-Posteriors]\label{Thm:CLT}
Under conditions C.1-C.3, the quasi-posterior density (\ref{Def:f}) approaches the normal density (\ref{Def:phi}) in the following sense:
$$\int_{\Lambda}
|f(\lambda) - \phi(\lambda)| d \lambda = o_p(1). $$
\end{theorem}

 Theorem \ref{Thm:CLT} is a simple preliminary result.
However, the result is essential for defining the environment in
which the main results of this paper -- the computational complexity
results -- will be developed. The theorem shows that in large
samples, provided that some regularity conditions hold, Bayesian and
quasi-Bayesian inference have good large sample properties. The
main part of the paper, namely Section \ref{Sec:Sampling}, develops the
\textit{computational implications} of the CLT conditions. In
particular, Section \ref{Sec:Sampling}  shows that polynomial time computing of
Bayesian and quasi-Bayesian estimators by MCMC is in fact implied
by the CLT conditions.   Therefore,  the CLT
conditions are essential for both good statistical properties of the posterior or quasi-posterior under increasing dimension, as shown in Theorem \ref{Thm:CLT}, and for good computational properties as shown in Section \ref{Sec:Sampling}.

By allowing increasing dimension ($d\to \infty$) Theorem
\ref{Thm:CLT} extends the CLT previously derived in the literature
for posteriors in the likelihood framework (Blackwell \cite{Blackwell}, Bickel and Yahav
\cite{BY}, Ibragimov and Hasminskii \cite{IH}, Bunke and Milhaud
\cite{BM}, Ghosal \cite{G2000}, Shen \cite{Shen}) and for
quasi-posteriors in the general extremum framework, when the
likelihood is replaced by general criterion functions (Blackwell \cite{Blackwell}, Liu, Tian, and Wei \cite{LTW} and Chernozhukov and Hong \cite{CH}). The
theorem also extends the results in Ghosal \cite{G2000}, who also
considered increasing dimensions but focused his analysis to the
exponential likelihood family framework. In contrast, Theorem
\ref{Thm:CLT} allows for non-exponential families and for
quasi-posteriors in place of posteriors. Recall that
quasi-posteriors result from using quasi-likelihoods and other
criterion functions in place of the likelihood. This substantially expands
the scope of the applications of the result.
Importantly, Theorem \ref{Thm:CLT} allows for non-smoothness and
even discontinuities in the likelihood and criterion functions,
which are pertinent in a number of applications listed in the
introduction.

\textbf{The Problem of the Paper.} Our problem is to characterize
the complexity of obtaining draws from $f(\lambda)$ and of Monte
Carlo integration for computing $$\int g(\lambda) f(\lambda) d\lambda,$$ where
$f(\lambda)$  is restricted to the approximate support $K$. The
procedure used to obtain the basic draws as well as to carry out
Monte Carlo integration is a Metropolis random walk, which is a
standard MCMC algorithm used in practice. The tasks are thus:

\begin{itemize}

\item[\textbf{I}.\ \ ] Characterize the complexity of sampling
from $f(\lambda)$ as a function of $(d, n, \epsilon_{1},
\epsilon_{2},
K)$;%

 \item[\textbf{II}.\ \ ] Characterize the complexity of calculating
$\int g(\lambda) f(\lambda) d\lambda$
as a function of $(d, n, \epsilon_{1}, \epsilon_{2},K)$;%

\item[\textbf{III}.\ ] Characterize the complexity of sampling
from $f(\lambda)$ and performing integrations with $f(\lambda)$ in
large samples as $d,n \to \infty$ by invoking the bounds on $(d,
n, \epsilon_{1}, \epsilon_{2}, K)$ imposed by the
CLT;%
\item[\textbf{IV}.] Verify that the CLT conditions are applicable
in a variety of statistical problems.
 \end{itemize}
This paper formulates and solves this problem. Thus,  the paper
brings the CLT restrictions into the complexity analysis and
develops complexity bounds for sampling and integrating from
$f(\lambda)$ under these restrictions. These CLT restrictions,
arising from the use of large sample theory and the imposition of certain
regularity conditions, limit the behavior of $f(\lambda) $ over
the approximate support set $K$ in a specific manner that allows
us to establish polynomial computing time for sampling and
integration. Because the conditions for the CLT do not provide
strong restrictions on the tail behavior of $f(\lambda)$ outside
$K$ other than C.1, our analysis of complexity is limited entirely
to the approximate support set $K$ defined in C.1-C.3.

By solving the above problem, this paper contributes to the recent
literature on the computational complexity of Metropolis
procedures. Early work was primarily concerned with the question of
approximating the volume of high dimensional convex sets where
uniform densities play a fundamental role (Lov\'asz and Simonovits
\cite{LS}, Kannan, Lov\'asz and Simonovits \cite{KLS,KLS97}).
Later, the approach was generalized for the cases where the
log-likelihood is concave (Frieze, Kannan and Polson \cite{FKP},
Polson \cite{P}, and Lov\'asz and Vempala \cite{LV01,LV02,LV03}).
However, under log-concavity the maximum likelihood or extremum
estimators are usually preferred over Bayesian or quasi-Bayesian
estimators  from a computational point of view. Regarding cases in which log-concavity is absent, the settings in which there is great practical
appeal for using Bayesian and quasi-Bayesian estimates, have received little treatment in the literature. One important exception is
the paper of Applegate and Kannan \cite{AK}, which covers
nearly-log-concave but smooth densities using a discrete Metropolis
algorithm. In contrast to Applegate and Kannan \cite{AK}, our approach allows
for both discontinuous and non-log-concave densities that are
permitted to deviate from the normal density (not from an
arbitrary log-concave density, like in Applegate and Kannan
\cite{AK}) in a specific manner. The manner in which they deviate
from the normal is motivated by the CLT and controlled by
parameters $\epsilon_1$ and $\epsilon_2$, which are in turn
restricted by the CLT conditions. Using the CLT restrictions also
allows us to treat non-discrete sampling algorithms. In fact, it
is known that the canonical Gaussian walk analyzed in Section
\ref{Sec:Geo} does not have good complexity properties (rapidly
mixing) for arbitrary log-concave density functions, see Lov\'asz
and Vempala \cite{LV03}. Nonetheless, the CLT conditions imply
enough structure so that even a canonical Gaussian walk becomes in fact
rapidly mixing. Moreover, the analysis is general in that it
applies to any Metropolis chain, provided that it satisfies a simple
geometric condition.  We illustrate this condition with the
canonical algorithm.  This suggests that the same approach can be
used to establish polynomial bounds for various more sophisticated
schemes. Finally, as is standard in the literature, we assume that
the starting point for the algorithm occurs in the approximate
support of the posterior. Indeed, the polynomial time bound that
we derive applies only in this case because this is the domain
where the CLT provides enough structure on the problem. Our
analysis does not apply outside this domain.

\section{The complexity of sampling using random walks}\label{Sec:Sampling}

\subsection{Set-Up and Main Result} In this section we bound the computational
complexity of obtaining  a draw from a random variable
approximately distributed according to a density function $f$ as
defined in (\ref{Def:f}). (Section \ref{Sec:Integration} builds
upon these results to study the associated integration problem.)
By invoking condition C.1, we restrict our attention entirely to
the approximate support set $K$ and the accuracy of sampling will
be defined over this set. Consider a measurable space
$(K,\mathcal{A})$. Our task is to draw a random variable according
to a density function $f$ restricted to $K$. This
density induces a probability distribution on $K$ defined by
 $Q(A) = \int_A f(x)dx/\int_Kf(x)dx$ for any $A \in \mathcal{A}$.
Asymptotically, it is well-known that random walks combined with a
Metropolis filter are capable of performing such a task. Such random
walks are characterized by an initial point $u_0$ and a {\it
one-step} probability distribution, which depends on the current
point, to generate the next candidate point of the random walk.
The candidate point is accepted with a probability given by the
Metropolis filter, which depends on the likelihood function $\ell$,
on the current and on the candidate point, and otherwise the random
walk stays at the current point (see Casella and Robert \cite{RC}
and Vempala \cite{VSurvey} for details; Section \ref{Sec:Geo}
describes the canonical Gaussian random walk).

In the complexity analysis of this algorithm we are interested in
bounding the number of steps of the random walk required to draw a
random variable from $Q$ with a given precision. Equivalently, we
are interested in bounding the number of evaluations of the local
likelihood function $\ell$ required for this purpose.

Next, following  Lov\'asz and
Simonovits \cite{LS} and Vempala \cite{VSurvey}, we review definitions of concepts relevant for our
analysis.  Let $q(x|u)$
denote the probability density to generate a
candidate point and $1_u(A)$ be the indicator function of the set
$A$. For each $u \in K$ the one-step distribution $P_u$, the
probability distribution after one step of the random walk
starting from $u$, is defined as
\begin{equation} P_u (A) = \int_{K \cap A}
\min\left\{\frac{f(x)q(u|x)}{f(u)q(x|u)},1\right\} q(x|u) dx +
(1-p_u) 1_{u}(A),
\end{equation} \noindent where \begin{equation} p_u =  \int_{K}
\min\left\{\frac{f(x)q(u|x)}{f(u)q(x|u)},1\right\} q(x|u)
dx\end{equation} is the probability of making a proper move,  namely the move to $x \in K, x \neq u$, after one step of the chain from $u \in K$.

The triple $(K,\mathcal{A},\{P_u: u\in K\})$, along with a
starting distribution $Q_0$, defines a Markov chain in $K$. We
denote by $Q_t$ the probability distribution obtained after $t$
steps of the random walk. A distribution $Q$ is called stationary
on $(K,\mathcal{A})$ if for any $A \in \mathcal{A}$,
\begin{equation} \int_K P_u(A) dQ(u) = Q(A). \end{equation}
Given the random walk described earlier, the unique stationary
probability distribution $Q$ is induced by the function $f$, $Q(A)
= \int_A f(x) dx/\int_Kf(x)dx$ for all $A \in \mathcal{A}$, see
e.g. Casella and Roberts \cite{RC}. This is the main motivation
for most of the MCMC studies found in the literature since it
provides an asymptotic method to approximate the density of
interest. As mentioned before, our goal is to properly quantify
this convergence and for that we need  to review additional
concepts.

The ergodic flow of a set $A$ with respect to a distribution $Q$
is defined as $$ \Phi(A) = \int_A P_u (K\backslash A)dQ(u).$$ It
measures the probability of the event $\{ u \in A, u' \notin A \}$
where $u$ is distributed according to $Q$ and $u'$ is distributed according to $P_u$; it captures the average flow of points leaving $A$ in one step of the random
walk. The measure  $Q$ is stationary  if and only if
$\Phi(A) = \Phi(K \backslash A)$ for all $A \in \mathcal{A}$ since
$$\begin{array}{rcl}
\Phi(A) & = &\displaystyle \int_{A} P_u(K\setminus A) dQ(u) =
\int_{A} (1 - P_u(A)) \ dQ(u) \\
\\
&= &\displaystyle Q(A) - \int_{A} P_u(A) dQ(u) = \int_K P_u(A)
dQ(u) - \int_{A} P_u(A) dQ(u) \\
\\
&= & \Phi(K\setminus A).\end{array}
$$ A Markov chain is said to be ergodic if $\Phi(A)>0$ for every
$A$ with $0 < Q(A) <1$, which is the case for the Markov chain
induced by the random walk described earlier due to the
assumptions on $f$, namely conditions C.1 and C.2.

 Next we recall the concept of a conductance of a Markov chain,
  which plays a key role in the convergence
analysis. Intuitively, a Markov chain will
converge slowly to the steady state if there exists a set $A$ in
which the Markov chain stays ``too long" relative to the measure
of $A$ or its complement $K\backslash A$. In order for a Markov
chain to stay in $A$ for a long time, the probability of stepping out
of $A$ with the random walk must be small, that is, the ergodic flow of $A$ must be small relative to the measures of $A$ and $K\backslash A$. The concept of conductance of a set $A$ quantifies this notion:
$$ \phi(A) = \frac{\Phi(A)}{\min\{Q(A),Q(K\backslash A)\}}, \ \ 0 < Q(A) < 1.$$
\noindent The global conductance of the Markov chain is the minimum conductance over sets with positive measure
\begin{equation} \phi = \inf_{A \in \mathcal{A}: 0 < Q(A)<1} \phi(A). 
 \end{equation}

Lov\'asz and Simonovits \cite{LS} proved the connection between conductance and
convergence for the continuous state space, and Jerome and Sinclair  \cite{JS:conf,JS:pub} proved
the connection for the discrete state space. We will extensively use Corolary 1.5 of Lov\'asz and Simonovits \cite{LS}, restated here as follows:
Let $Q_0$ be $M$-warm with respect to the stationary
distribution $Q$, namely
\begin{equation}\label{Def:M} \sup_{A \in
\mathcal{A}: Q(A) > 0} \frac{Q_0(A)}{Q(A)} =M,
\end{equation}
 then, the total variation distance between the stationary distribution $Q$ and the distribution $Q_t$, obtained after $t$ steps of the Markov chain starting from $Q_0$, is bounded above by a function of global conductance $\phi$ and warmness parameter $M$:
\begin{equation}\label{LS}
 \| Q_t - Q \|_{TV} = \sup_{A \in \mathcal{A}} |Q_t(A) - Q(A)| \leq \sqrt{M} \left( 1 - \frac{\phi^2}{2} \right)^t.
 \end{equation}

 Therefore, the global conductance $\phi$ determines the number of
steps required to generate a random point whose distribution $Q_t$ is
within a specified distance of the target distribution $Q$. The
conductance $\phi$ also bounds the autocovariance between
consecutive elements of the Markov chain,  which is important for
analyzing the computational complexity of integration by MCMC;
see Section \ref{Sec:Integration} for a more detailed
discussion.  The warmness parameter $M$, which measures how
the starting distribution $Q_0$ differs from the target distribution $Q$,
also plays an important role in determining the quality of convergence of $Q_t$ to $Q$. In what follows, we will calculate $M$ explicitly for the canonical random walk.

The main result of this paper provides a lower bound for the
global conductance of the Markov chain $\phi$ under the CLT
conditions. In particular, we show that $1/\phi$ is bounded by a
fixed polynomial in the dimension of the parameter space even for
a canonical random walk considered in Section
\ref{Sec:Geo}. In order to show this, we require the
following geometric condition on the difference between the
one-step distributions.
\begin{itemize}
\item[\textbf{D.1}]{\em There exist positive sequences $h_n$ and
$c_n$ such that for every $u, v \in K$, $\|u-v\| \leq h_n$ implies
that}
$$ \| P_u - P_v \|_{TV} < 1 - c_n. $$
\item[\textbf{D.2}] {\em The sequences above can be taken to
satisfy the following bounds}

$$ \frac{1}{ c_n \min\{h_n\sqrt{\lambda_{min}},1\}} = O_p(d).
$$

\end{itemize}

Condition D.1 holds if at least a $c_n$-fraction of the probability
distribution associated with $P_u$ varies smoothly as the
point $u$ changes. Condition D.2 imposes a particular rate for the
sequences. As shown in Theorem
\ref{Thm:Main} below, the rates in Conditions D.1 and D.2 play an important role in delivering good, that is, polynomial time, computational complexity. We show in Section \ref{Sec:Geo} that Conditions D.1 and D.2 hold for the canonical Gaussian
walk under Conditions C.1, C.2, and C.3. with $$
 1/h_n = O_p(d) \ \ \mbox{and} \ \ 1/c_n = O_p(1),
$$
and $\lambda_{min}$ bounded away from zero. Moreover, the rates in Condition D.2 appear to be sharp for the canonical Gaussian walk under our framework. It remains an important question whether different types of random walks could lead to better rates  than those in Condition D.2 (see Vempala \cite{VSurvey} for a relevant survey). Another interesting question is the establishment of lower bounds on the computational complexity of the type considered in Lov\'asz \cite{L-HR}.

Next we state the main result of the section.

\begin{theorem}[Main Result on Complexity of Sampling]\label{Thm:Main}  Under
Conditions  C.1, C.2, and D.1, the global conductance of the
induced Markov chain satisfies
\begin{equation}\label{Eq:phi}{ 1/\phi = O\left(
\frac{ e^{2(\epsilon_1 + \epsilon_2
\|K\|_J^2/2)}}{c_n \min\{ h_n\sqrt{\lambda_{\min}}, 1 \}}
 \right)}. \end{equation} In particular, a random walk satisfying these assumptions requires at most
\begin{equation} \label{Eq:first bound} N_\varepsilon = O_p\( e^{4(\epsilon_1 + \epsilon_2 \|K\|_J^2/2)}    \
\frac{ \ \ln (M/\varepsilon)}{ (c_n \min\{ h_n\sqrt{\lambda_{min}}, 1
\})^2} \)
 \end{equation} steps to achieve $\|
Q_{N_\varepsilon} - Q \|_{TV} \leq \varepsilon$
where $Q_0$ is $M$-warm with respect to $Q$.
Finally, if Conditions C.1, C.2, C.3,
D.1 and D.2 hold, we have that
$$1/\phi = O_p( d )$$ and the number of steps
{\bf $N_\varepsilon$} is bounded by
\begin{equation}\label{Eq:Fbound} O_p \( d^2 \ \ln (M/\varepsilon) \).
\end{equation}
\end{theorem}

Thus, under the CLT conditions, Theorem \ref{Thm:Main} establishes the polynomial bound on the computing time, as stated in equation (\ref{Eq:Fbound}). Indeed, CLT conditions C.1 and C.2 first lead to the bound (\ref{Eq:first bound}) and, then, condition C.3, which imposes  $\epsilon_1 = o_p(1)$ and $\epsilon_2 \cdot \|K\|_J^2 = o_p(1)$,  leads to the polynomial bound (\ref{Eq:Fbound}).  It is also useful to note that, if the stated CLT conditions do not hold, the bound on the computing time needs not be polynomial: in particular, the first bound (\ref{Eq:first bound}) is exponential in $\epsilon_1$ and $\epsilon_2\|K\|^2_J$.  It is also useful to note that the approximate normality of posteriors and quasi-posteriors implied by the CLT conditions plays an important role in the proofs of this main result and of auxiliary lemmas.  Therefore, the CLT conditions are essential for both (a) good statistical properties of the posterior or quasi-posterior under increasing dimension, as shown in Theorem \ref{Thm:CLT} and (b) for good computational properties, as shown in Theorem \ref{Thm:Main}. Thus, results (a) and (b) establish a clear link between the computational properties and the statistical environment.

The relevance of the particular random walk in bounding the
conductance is captured through the parameters $c_n$ and $h_n$
defined in condition D.1. Theorem \ref{Thm:Main} shows that
as long as we can take $1/c_n$ and $1/h_n$ to be bounded by a
polynomial in the dimension of the parameter space $d$, we will
obtain polynomial time guarantees for the sampling problem. In
some cases, the warmness parameter $M$ appearing in (\ref{Eq:Fbound}) can also be related to the particular random walk being used. This is the case in the canonical random walk discussed in detail in
Section \ref{Sec:Geo}.

\subsection{Proof of the Main Result}\label{Sec:ProofMain}

The proof of Theorem \ref{Thm:Main} relies on a new
iso-perimetric inequality (Corollary 1) and a geometric property of the particular random walk
(condition D.1). After the connection between the
iso-perimetric inequality and the ergodic flow is established, the geometric property allows us to use the first result to bound the
conductance from below.  In what follows we provide an outline of
the proof, auxiliary results, and, finally, the formal proof.

\subsubsection{Outline of the Proof}  The proof follows the arguments in
Lov\'asz and Simonovits \cite{LS} and Lov\'asz and Vempala \cite{LV01}. In order to bound the ergodic flow of $A \in
\mathcal{A}$, consider the particular disjoint partition
$K=\widetilde S_1 \cup \widetilde S_2 \cup \widetilde S_3$ where
$\widetilde S_1 \subset A$, $\widetilde S_2 \subset K\setminus A$,
and $\widetilde S_3$ consists of points in $A$ or $K\setminus A$
for which the one-step probability of going to the other set is at
least $c_n/2$ (to be defined later). Therefore we have
$$
\begin{array}{rl}
\Phi(A) & =  \int_{A} P_u(K\setminus A) dQ(u) = \frac{1}{2}
\int_{A} P_u(K\setminus A) dQ(u) + \frac{1}{2} \int_{K\setminus A} P_u(A) dQ(u)\\
& \geq  \frac{1}{2}\int_{\widetilde S_1} P_u(K\setminus A) dQ(u) +
\frac{1}{2} \int_{\widetilde S_2} P_u(A) dQ(u) + \frac{c_n}{4}
Q(\widetilde S_3).
\end{array}
$$ where the second equality holds because $\Phi(A) = \Phi(K\setminus
A)$.

Since the first two terms could be arbitrarily small, the result
will follow by bounding the last term from below. This will be
achieved by a new iso-perimetric inequality tailored to the
CLT framework and derived in Section \ref{Sec:LogConc}. This
result will provide a lower bound on $Q(\tilde S_3)$, which is
increasing in the distance between $\widetilde S_1$ and
$\widetilde S_2$.

Therefore, it remains to show that the distance
between $\widetilde S_1$ and
$\widetilde S_2$ is suitably bounded below. This follows
from the geometric property stated in condition D.1.
Given two points $u \in \widetilde S_1$ and $v \in \widetilde
S_2$, we have $P_u(K\setminus A) \leq c_n/2$ and $P_v(A)\leq
c_n/2$. Therefore, the total variation distance between their
one-step distributions is bounded as $$\| P_u - P_v \|_{TV} \geq |P_u(A) - P_v(A)|
\geq 1 - c_n.$$ In such a case, condition D.1 implies that the distance $\|u-v\|$
is bounded from below by $h_n$. Since $u$ and $v$ are arbitrary
points, the distance between sets $\widetilde S_1$ and $\widetilde S_2$ is bounded below by $h_n$.

This leads to a lower bound for the global conductance. After
bounding the global conductance from below, Theorem \ref{Thm:Main} follows
by invoking the conductance theorem of  \cite{LS} restated in equation (\ref{LS}) and the CLT conditions.

\subsubsection{An Iso-perimetric
Inequality}\label{Sec:LogConc}

We start by defining  a notion of approximate log-concavity. A
function $f:\RR^d \to \RR $ is said to be log-$\beta$-concave if
for every $\alpha \in [0,1]$, $x,y \in \RR^d$, we have
$$f\left(\alpha x + (1-\alpha) y \right) \geq \beta
f(x)^{\alpha} f(y)^{1-\alpha} $$ for some $\beta \in (0,1]$, and $f$
is said to be log-concave if $\beta$ can be taken to be one. The
class of log-$\beta$-concave functions is rather broad, including, for
example, various non-smooth and discontinuous functions.

This concept is relevant under our CLT conditions C.1-C.3, since the relations (\ref{Def:g}) and (\ref{Def:m}) imposed by these conditions imply the following:

\begin{lemma}\label{Lemma1} Over the set $K$, the functions $f(\lambda) := \ell(\lambda)/\int_{\Lambda}
\ell(\lambda) d\lambda$ and $\ell(\lambda)$ can be written as the product of a
Gaussian function, $e^{-\frac{1}{2}\lambda'J\lambda}$, and a
log-$\beta$-concave function with parameter
$$
\beta = e^{-2(\epsilon_1 +\epsilon_2\|K\|_J^2/2)}.
$$
\end{lemma}


The representation of Lemma \ref{Lemma1}
gives us a convenient structure to establish the following
iso-perimetric inequality.

\begin{lemma}\label{betaISO}
Consider any measurable partition of the form $K = S_1 \cup S_2
\cup S_3$ such that the distance between $S_1$ and $S_2$ is at
least $t$, i.e. $d(S_1,S_2) \geq t$. Let $Q(S) = \int_S f dx /
\int_K f dx $. Then for any lower semi-continuous function $f(x) =
e^{-\|x\|^2} m(x)$, where $m$ is a log-$\beta$-concave function,
we have
$$Q(S_3) \geq \beta \frac{2 t e^{-t^2/4}}{\sqrt{\pi}} \min
\left\{ Q(S_1), Q(S_2) \right\}.$$
\end{lemma}

 The iso-perimetric inequality of Lemma \ref{betaISO} states that if two subsets
of $K$ are far apart, the measure of the remaining subset of $K$ should
be comparable to the measure of at least one of the original
subsets.  This iso-perimetric inequality extends the
iso-perimetric inequality in Kannan and Li
\cite{KLi}. The proof builds on
their proof as well as on the ideas in Applegate
and Kannan \cite{AK}. Unlike the inequality in
 Kannan and Li \cite{KLi}, Lemma \ref{betaISO} removes
the smoothness assumptions on $f$,
covering both non-log-concave and discontinuous
cases.

The following corollary extends Lemma \ref{betaISO} to the case of an arbitrary covariance matrix $J$.

\begin{corollary}[Iso-perimetric Inequality]\label{Corollary:Iso}
Consider any measurable partition of the form $K = S_1 \cup S_3
\cup S_2$ such that $d(S_1,S_2) \geq t$, and let $Q(S) = \int_S f
dx / \int_K f dx $. Then, for any lower semi-continuous function
$f(x)$ $=$ $e^{-\frac{1}{2} x'Jx} m(x)$, where $m$ is a
log-$\beta$-concave function and $J$ is positive definite covariance matrix, we have
$$Q(S_3) \geq \beta \ \sqrt{\lambda_{min}}t e^{-\lambda_{min}t^2/8} \ \sqrt{\frac{2 }{\pi}} \min
\left\{ Q(S_1), \ Q(S_2) \right\},$$ where $\lambda_{min}$ denotes
the minimum eigenvalue of $J$.
\end{corollary}

\subsubsection{Proof of Theorem \ref{Thm:Main}} Fix an arbitrary set $A\in \mathcal{A}$ and denote by
$A^c=K\setminus A$ the complement of $A$ with respect to $K$. We
will prove that
\begin{equation}\label{MixingProof} \Phi(A) \geq \frac{ c_n}{4} \beta  \sqrt{\frac{ 2}{\pi e}} \min\left\{ \frac{h_n}{2} \sqrt{\lambda_{min}}, 1 \right\}    \min\{Q(A), Q (A^c)\}, %
\end{equation}
\noindent where $\beta = e^{-2(\epsilon_1 +\epsilon_2\|K\|_J^2/2)}$ is as defined in Lemma \ref{Lemma1}. This result implies the desired bound on the global
conductance $\phi$.

Consider the following auxiliary definitions: $$
\widetilde S_1 = \displaystyle \left\{ u \in A : P_u(A^c) < \frac{c_n}{2} \right\},  %
\widetilde S_2 = \displaystyle \left\{ v \in A^c : P_v(A) < \frac{c_n}{2} \right\},  %
\widetilde S_3 = \displaystyle  K \backslash ( \widetilde S_1 \cup
\widetilde S_2 ).
$$\noindent  In this case  $Q (\widetilde S_1 ) \leq Q
(A)/2$, we have
$$ \begin{array}{rcl}
\Phi(A)& =& \displaystyle \int_{A} P_u(A^c) dQ(u) \geq
\displaystyle \int_{A \backslash \widetilde S_1} P_u(A^c) dQ(u)
\geq \displaystyle \int_{A\backslash \widetilde S_1}
\frac{c_n}{2} dQ(u) \\
\\
&\geq & \displaystyle \frac{c_n}{2}  Q (A\backslash \widetilde S_1) \geq \frac{c_n}{4} Q(A),\\
\end{array}
$$
\noindent which immediately implies the inequality (\ref{MixingProof}).  In the case $Q(\widetilde S_2) \leq Q(A^c)/2$, we apply a similar argument.

In the remaining case  $Q(\widetilde S_1) \geq Q(A)/2$ and $
Q(\widetilde S_2) \geq Q(A^c)/2$, we proceed as follows. Since $\Phi(A) = \Phi(A^c)$ we
have that
$$
\begin{array}{rcl}
\displaystyle \Phi(A) = \int_{A} P_u(A^c) dQ(u) & = & \frac{1}{2}
\int_{A} P_u(A^c) dQ(u) + \frac{1}{2} \int_{A^c} P_v(A) dQ(v)\\
&\geq &  \frac{1}{2} \int_{A\setminus \widetilde S_{1}} P_u(A^c)
dQ(u) + \frac{1}{2} \int_{A^c\setminus \widetilde S_{2}} P_v(A) dQ(v)\\
 & \geq &
\frac{1}{2}\int_{\widetilde S_3} \frac{c_n}{2} dQ(u) =  \frac{c_n}{4} Q(\widetilde S_3),\\ %
\end{array}
$$ where we used that $\widetilde S_3 = K\setminus (\widetilde S_1 \cup \widetilde S_2 ) = (A
\setminus \widetilde S_1)\cup (A^c\setminus \widetilde S_2)$.
Given the definitions of the sets $\widetilde S_1$ and $\widetilde
S_2$, for every $u \in \widetilde S_1$ and $v \in \widetilde S_2$
we have $$ \| P_u - P_v \|_{TV} \geq P_u(A) - P_v(A) =  1 -
P_u(A^c) - P_v(A) \geq 1 - c_n.
$$ \noindent In such a case, by condition D.1, we have that $
\|u-v\| > h_n$ for every $u \in \widetilde S_1$ and $v \in
\widetilde S_2$. Thus, we can apply the iso-perimetric inequality
of Corollary \ref{Corollary:Iso}, with $d(\widetilde S_1,
\widetilde S_2) \geq h_n$, to bound $Q(\widetilde S_3)$. We then obtain
$$
\begin{array}{rcl}
\displaystyle \int_{A} P_u(A^c) dQ(u)
& \geq & \max_{0 \leq t \leq h_n}\frac{c_n}{4} \beta \sqrt{\frac{2 }{\pi}} \sqrt{\lambda_{\min}} \ t e^{-\frac{1}{8}\lambda_{min}t^2}\ \min\{Q(\widetilde S_1),Q(\widetilde S_2)\}\\
& \geq &  \frac{c_n}{4}\beta  \sqrt{\frac{2}{\pi e}} \min\left\{ \frac{h_n}{2} \sqrt{\lambda_{min}}, 1 \right\} \min\{Q(A),Q(A^c)\}.\\
\end{array}
$$where  we used the fact that $\max_{0 \leq t \leq h_n} \sqrt{\lambda_{\min}} t e^{-\frac{1}{8}\lambda_{min}t^2}$ is bounded below by $  \min\left\{ h_n \sqrt{\lambda_{min}}, 2 \right\} e^{-1/2} $ and that $\min\{Q(\widetilde S_1),Q(\widetilde S_2)\} \geq \min\{Q(A),Q(A^c)\}/2$.
Thus, the inequality (\ref{MixingProof}) and the lower bound
on conductance (\ref{Eq:phi}) follow.

The bound (\ref{Eq:first bound}) on the number of steps of the Markov Chain follows from the lower bound on conductance (\ref{Eq:phi}) and the conductance theorem of \cite{LS} restated in equation (\ref{LS}).  The remaining results in Theorem \ref{Thm:Main} follow by invoking the CLT
conditions.$\qed$


\subsubsection{The case of the Gaussian random walk}\label{Sec:Geo}

In order to provide a concrete example of our
complexity bounds, we consider the canonical
random walk induced by a Gaussian distribution.
Such a random walk is completely characterized by
an initial point $u_0$,  a fixed standard
deviation $\sigma>0$, and its {\it one-step}
move. The latter is defined by the procedure of
drawing a point $y$ from a Gaussian
distribution centered at the current point $u$
with covariance matrix $\sigma^2I$, and then if $y \in K$ moving to $y$ with
probability
$\min\{f(y)/f(u),1\}=\min\{\ell(y)/\ell(u),1\}$,
and otherwise staying at $u$.

We start with the following auxiliary result.

\begin{lemma}\label{GeoLips}
Let $a:\RR^n \to \RR$ be a function such that $\ln a$ is Lipschitz
with constant $L$ over a compact set $K$. Then, for every $u \in K$
and $r>0$, $$ \inf_{y \in B(u,r) \cap K}
\[a(y)/a(u)\] \geq  e^{-Lr}. $$
\end{lemma}

Given the ball $K=B(0, \|K\|)$, we can bound the Lipschitz constant of
the function $ -\lambda' J\lambda/2$ by
\be\label{Def:L} L =
\sup_{\lambda \in K} \| J \lambda \| = \lambda_{max}\|K\|.
\ee

We define
the parameter $\sigma$ of the Gaussian random walk as \be\label{Def:sigma}
\displaystyle \sigma = \min \left\{ \frac{1}{4\sqrt{d}L}, \frac{\|K\|}{120d}\right\}.%
\ee Using (\ref{Def:L}) and that $\|K\| > \sqrt{d/\lambda_{\min}}$ it
follows that \be\label{Eq:Bound:sigma} \sigma
\geq\frac{1}{120\lambda_{max}\sqrt{d}\|K\|}. \ee

In order to apply Theorem \ref{Thm:Main} we rely on $\sigma$ being
defined in (\ref{Def:sigma}) as a function of the relevant
theoretical quantities. More practical choices of the parameter,
as in Robert and Rosenthal \cite{RR} and Gelman, Roberts and Gilks
\cite{GRG}, suggest that we tune the parameter to ensure a
particular average acceptance rate for the steps of the Markov
Chain. These cases are exactly the cases covered by our
(theoretical) choice of $\sigma$ (of course, different constant
acceptance rates lead to different constants in the proof of
the theorem). Moreover, a different choice of covariance matrix
for the auxiliary Gaussian distribution can lead to improvements
in practice but, under the assumptions on the matrix $J$, does not
affect the overall dependence on the dimension $d$, which is our
focus here.

Next we verify conditions D.1 and D.2 for the Gaussian random
walk. Although this approach follows that in Lov\'asz and
Vempala \cite{LV01,LV02,LV03}, there are two important differences
which call for a new proof. First, we no longer rely on the
log-concavity of $f$. Second, we use a different random walk.

\begin{lemma}\label{GeoProb}
Let $u,v \in K := B(0,\|K\|)$, suppose that $\sigma \leq \min\{\frac{1}{4\sqrt{d}L}, \frac{\|K\|}{120d} \}$,
and $ \|u- v\| < \frac{\sigma}{8}$, where $L$ is the
Lipschitz constant specified in equation (\ref{Def:L}). Under conditions
C.1-C.2, we have for $\beta=e^{-2(\epsilon_1 +\epsilon_2\|K\|_J^2/2)}$ that
$$\| P_u
- P_v \|_{TV} \leq 1 - \frac{\beta}{3e}. $$

\end{lemma}

\begin{remark}
Therefore, the Gaussian random walk
satisfies condition D.1 with
\begin{equation}\label{Cte:G}c_n = \frac{\beta}{3e} \ \
\mbox{ and } \ \ h_n = \frac{\sigma}{8}.\end{equation} Under the CLT
framework, i.e. conditions C.1, C.2, and C.3, we have that $c_n$
and $h_n$ as defined in (\ref{Cte:G}) satisfy condition D.2 with $$1/h_n = O_p(d) \text{ and } 1/c_n = O_p(1),$$ and $\lambda_{min}$ bounded away from zero.

By applying Theorem \ref{Thm:Main} to the
Gaussian random walk, the conductance bound
(\ref{Eq:phi}) becomes
$$1/\phi = O\left(  \frac{\lambda_{max}}{\lambda_{min}} \ d \ e^{2(\epsilon_1+\epsilon_2\|K\|_J/2)} \right)
=
O_p(d)$$ and the bound on the number of steps $N_\varepsilon$ in
(\ref{Eq:first bound}) becomes
\begin{equation}\label{Eq:Fbound:G}
O_p\big(d^2\ln ( M /\varepsilon ) \big).\end{equation}
\end{remark}

Next we discuss and bound the dependence on $M$, the ``distance"
of the initial distribution $Q_0$ from the stationary distribution
$Q$ as defined in (\ref{Def:M}). A natural candidate for a
starting distribution $Q_0$ is the one-step distribution
conditional on a proper move from an arbitrary point $u\in K$. Thus,
$$Q_{0} (A) =  p_u^{-1} \cdot \int_{K \cap A}
\min\left\{\frac{f(x)q(u|x)}{f(u)q(x|u)},1\right\} q(x|u) dx,
$$ where $$p_u = \int_{K}
\min\left\{\frac{f(x)q(u|x)}{f(u)q(x|u)},1\right\} q(x|u)
dx$$ is the probability of a proper move, namely the move to $x \in K, x \neq u$, after one
step of the chain from $u \in K$. We
emphasize that, in general, such choice of $Q_0$ could lead to
values of $M$ that are arbitrary large. In fact, this could happen
even in the case of the stationary density being a uniform
distribution on a convex set (see Lov\'asz and Vempala
\cite{LV03}). However, this is not the case under the CLT
framework as shown by the following lemma.

\begin{lemma}\label{InitialDraw} 
Suppose conditions C.1-C.2 hold, then for $\beta=e^{-2(\epsilon_1 +\epsilon_2\|K\|_J^2/2)}$ we have that with a probability $p_u \geq \beta/(3e)$ the random walk makes a proper move.  Moreover, let $u \in K$ and $Q_0$ be the associated one-step distribution
conditional on performing a proper move starting from $u$, then $Q_0$
 is  $M$-warm with respect to $Q$, where
$$
\ln M =  O ( d \ln (\|K\|_J^2) + \|K\|^2_J + \epsilon_1 +
\epsilon_2 \|K\|^2_J ).
$$
Under conditions $\epsilon_1 = o_p(1)$, $ \epsilon_2\|K\|_J = o_p(1)$, and
$\|K\|_J = O(\sqrt{d})$  we have
$$ \ln M = O_p (d \ln d) \text{ and } p_u \geq  1/(3e) +o_p(1).$$
\end{lemma}

\begin{remark}[Overall Complexity for Gaussian Walk]
 The combination of this result with relation (\ref{Eq:Fbound:G}), which was derived from Theorem \ref{Thm:Main}, yields the overall (burn-in
plus post burn-in) running time $$ O_p(d^3 \ln  d ).$$
\end{remark}


\section{The complexity of Monte Carlo integration}\label{Sec:Integration} This section considers our second task
of interest -- that of computing a high dimensional integral of a
bounded real valued function $g$:
\begin{equation}\label{IntDef}
\mu_g = \int_K g(\lambda) dQ(\lambda).
\end{equation}
Theorem \ref{Thm:Main} showed that the CLT conditions provide enough structure to bound the conductance of the Markov chain associated with a particular random walk.
Below we also show how the conductance and CLT-based bounds on conductance impact the computational complexity of calculating (\ref{IntDef}) via standard schemes (long run, multiple runs, and subsampling). These new characterizations complement the
previous well-known characterizations of the error in estimating
(\ref{IntDef}) in terms of the covariance functions of the underlying
chain (Geyer \cite{Geyer}, Casella and Roberts \cite{RC}, and Fishman \cite{fish}).

In what follows, a random variable $\lambda^{t}$ is distributed according to $Q_{t}$, the probability measure obtained after iterating the chain $t$ times, beginning from a starting measure $Q_0$. The chain $\lambda^t, t=0, 1, ...$ has the stationary distribution $Q$. Accordingly, a standard estimate of (\ref{IntDef}), called the long-run (lr) average, takes the form
\begin{equation}\label{IntAppr} \hat \mu_g =
\frac{1}{{N}}\sum_{i=B}^{ B+N} g(\lambda^i),
\end{equation}
discarding the first $B$ draws, the burn-in sample, and using subsequent $N$ draws of the Markov chain.

The dependent nature of
the chain increases the number of post-burn-in draws $N$ needed to achieve a desired
precision compared to the infeasible case of independent draws
from $Q$.  It turns out that, as in the preceding analysis, the
conductance of the Markov chain is crucial
for determining the appropriate $N$.

The starting point of our analysis is a central limit theorem for
reversible Markov chains due to Kipnis and Varadhan \cite{KV}: Consider a reversible
Markov chain on $K$ with a stationary distribution $Q$. The lag $k$
autocovariance of the stationary time series $g(\lambda^i), i =1,2,...$, obtained by starting the Markov chain with
the stationary distribution $Q$ is defined as
$$ \gamma_k = {\rm Cov}_Q\left(g(\lambda^i),g(\lambda^{i+k})\right). $$
Then, for a stationary, irreducible, reversible Markov chain,
\begin{equation}\label{ThmKV}
 { N} E[(\hat \mu_g - \mu_g)^2] \to \sigma^2_g = \sum_{k=-\infty}^{+\infty} \gamma_k,
 \end{equation}

\noindent almost surely. If $\sigma^2_g$ is finite, then
 \begin{equation}\label{ThmKV1}
\sqrt{N}
(\hat \mu_g - \mu_g) \to_d N(0,\sigma^2_g).
 \end{equation}

In our case, $\gamma_0$ is finite since $g$ is bounded. Let us recall a
result, which is due to Lov\'asz and Simonovits \cite{LS}, and which states that
$\sigma^2_g$ can be bounded using the global conductance $\phi$ of
a stationary, irreducible, reversible Markov chain: Let $g$ be a
square integrable function with respect to the
stationary measure $Q$, then
 \begin{equation}\label{Cor:Cov}
|\gamma_k| \leq \left( 1 - \frac{\phi^2}{2}\right)^{|k|} \gamma_0 \ \ \mbox{and} \ \ \sigma^2_g \leq  \gamma_0
\left(\frac{4}{\phi^2}\right).
\end{equation}
We will use these conductance-based bounds to obtain bounds
on the complexity of integration under the CLT conditions.

There exist other methods for constructing the
sequence of draws in constructing estimators of the type (\ref{IntAppr}); we refer to Geyer \cite{Geyer} for a detailed discussion. In addition to the long run (lr) method, we also
consider  the subsample (ss) and multi-start (ms) methods.
Denote the number of post burn-in draws corresponding to each method as $N_{lr}$,
$N_{ss}$, and $N_{ms}$. As mentioned above, the long run method consists of generating
the first point using the starting distribution $Q_0$ and, after the
burn-in period, selecting the $N_{lr}$ subsequent points to
compute the sample average. The subsample method
also uses only one sample path, but the $N_{ss}$ draws  used in
the sample average are spaced out by $S$ steps of
the chain. Finally, the multi-start method uses $N_{ms}$ different
sample paths, initializing each one independently from the
starting probability distribution $Q_0$
and picking the last draw in each
sample path after the burn-in period to compute the average.
Thus, all estimators discussed above take the form
$$\hat \mu_{g} = \frac{1}{N}\sum_{i=1}^N
g(\lambda^{i,B})$$
with the underlying sequence $\lambda^{1,B},
\lambda^{2,B},...,\lambda^{N,B}$ produced as follows:
\begin{itemize}

\item  for lr,  $\lambda^{i,B} = \lambda^{i+B}$, where $B$ is the burn-in period,

\item for ss, $\lambda^{i,B} = \lambda^{iS+B}$, where $S$ is the
number of draws being skipped,

\item  for ms, $\lambda^{i,B}$ are i.i.d. draws from $Q_B$,
that is, $\lambda^{i,B} \sim \lambda^B$ for every $i$.

\end{itemize}

There is a final issue that must be addressed. Both the central limit
theorem of \cite{KV}, restated in equations (\ref{ThmKV}) and
(\ref{ThmKV1}) and the  conductance-based bound of \cite{LS} on
covariances restated in equation (\ref{Cor:Cov}) require that the
initial point be drawn from the stationary distribution $Q$.
However, we are starting the chain from some other distribution
$Q_0$, and in order to apply these results we need to first run
the chain for sufficiently many steps $B$,
 to bring the distribution of the draws $Q_B$ close to $Q$ in total variation
 metric.  This is what we call the burn-in period. However, even after
 the burn-in period there is still a discrepancy between $Q$ and $Q_B$,
 which should be taken into account.  But once $Q_B$ is close to $Q$, we can
 use the results on complexity of integration where sampling starts with $Q$
 to bound the complexity of integration where sampling starts with $Q_B$, where
 the bound depends on the discrepancy between $Q_B$ and $Q$.  Thus, our computational complexity calculations take into account all of the following
  three facts: (i) we are starting with a
distribution $Q_0$ that is $M$-warm with respect to $Q$, (ii) from
$Q_0$ we are making $B$ steps with the chain in the burn-in
period to obtain $Q_B$ such that $\|Q_B - Q\|_{TV}$ is
sufficiently small, and (iii) we are only using draws after the
burn-in period to approximate the integral.

We use the mean square error as the measure of
closeness for a consistent estimator:
$$ MSE( \hat \mu_g ) = E\[ \(\hat \mu_g - \mu_g\)^2\].$$

\begin{theorem}[Complexity of Integration]\label{NNN}
Let $Q_0$ be $M$-warm with respect to $Q$, and let $ \bar
g := \sup_{\lambda \in K} |g(\lambda)|$.
In order to obtain $$MSE(\hat \mu_g)
<\varepsilon$$ it is sufficient to use the following lengths of
the burn-in sample, $B$, and post-burn-in samples, $N_{lr},
N_{ss}, N_{ms}$:
$$\displaystyle B = \left( \frac{2}{\phi^2} \right) \ln
\left(\frac{24\sqrt{M} \bar g^2}{\varepsilon} \right)$$ and
 $$\displaystyle N_{lr} =
\frac{\gamma_0}{\varepsilon} \frac{6}{\phi^2},\ \
 \displaystyle N_{ss} = \frac{3\gamma_0}{\varepsilon} \ \mbox{with} \ S = \frac{2}{\phi^2}\ln \left(\frac{6\gamma_0}{\varepsilon}\right), \ \
 \displaystyle N_{ms}= \frac{2\gamma_0}{3\varepsilon}.$$
The overall complexities of the lr, ss, and ms methods are thus  $B +
N_{lr}$, $B + S N_{ss}$, and $B\times N_{ms}$.
\end{theorem}

For convenience, Table \ref{Table:Cond} tabulates the bounds for
the three different schemes. Note that the dependence on $M$ and $
\bar g$ is only via log terms.   Although the optimal choice of
the method depends on the particular values of the constants, when
$\varepsilon \searrow 0$, the long-run algorithm has the smallest
(best)  bound, while the the multi-start algorithm has the largest
(worst)  bound on the number of iterations.
\begin{table}
\begin{center}
{\small \caption{Burn-in and Post Burn-in Bounds on the Complexity of Integration
of a Bounded Function via
Conductance}\label{Table:Cond}
\begin{tabular}{llll}
\hline \hline \\
Method         &  Quantities   &    Complexity     \\
\hline \\
Long Run       &  $B + N_{lr}$    & $ \frac{2}{\phi^2} \left(
\ln\left( \frac{24\sqrt{M} \bar g^2}{\varepsilon} \right)\right)
 + \frac{2}{\phi^2} \left (\frac{3\gamma_0}{\varepsilon} \right)$  \\
Subsample      &  $B + N_{ss} \cdot S $    & $\frac{2}{\phi^2} \left( \ln\left( \frac{24\sqrt{M} \bar g^2}{\varepsilon} \right)\right) + \frac{2}{\phi^2} \left( \frac{3\gamma_0}{\varepsilon}\ln\left(\frac{24 \gamma_0}{\varepsilon}\right) \right)$   \\
Multi-start    &  $B \times N_{ms} $   & $ \frac{2}{\phi^2} \left(  \ln \left( \frac{24\sqrt{M} \bar g^2}{\varepsilon} \right)\right) \times \frac{2\gamma_0}{3\varepsilon}  $  \\
\hline\hline\\
\end{tabular}
}
\end{center}
\end{table}
Table \ref{Table:Complexity} presents the computational complexities implied by the
CLT conditions, namely $$\|K\|_J = O(\sqrt{d}), \epsilon_1 = o_p(1), \text{
and } \epsilon_2 \|K\|^2_J = o_p(1),$$ and the Gaussian random walk
studied in Section \ref{Sec:Geo}. The table assumes $\gamma_0$ and
$\bar g$ are constant, though it is straightforward to tabulate
the results for the case where $\gamma_0$ and $\bar g$ grow at
polynomial speed with $d$. Finally, note that the bounds apply
under a slightly weaker condition than the CLT requires, namely
that $\epsilon_1 =O_p(1) $ and $\epsilon_2 \|K\|^2_J = O_p(1)$.

\begin{table}
\begin{center}
{\small \caption{ Burn-in and Post Burn-in Bounds on the Complexity of Integration
of a Bounded Function using the
Gaussian random walk under the CLT
framework with $\|K\|_J = O(\sqrt{d}), \epsilon_1 = o_p(1),  \epsilon_2 \|K\|^2_J = o_p(1)$, and $\bar g = O(1)$. }\label{Table:Complexity}
\begin{tabular}{lcl}
\hline \hline \\
Method         &  Burn-in Complexity  &    Post-burn-in  Complexity      \\
\hline \\
Long Run       &  $O_p( d^3\ln d \cdot \ln \varepsilon^{-1} )$    &   $+ \ O_p( d^2 \cdot \varepsilon^{-1}  ) $   \\
Subsample      &  $O_p( d^3\ln d \cdot \ln \varepsilon^{-1} )$    &   $+ \ O_p( d^2 \cdot \varepsilon^{-1} \cdot \ln \varepsilon^{-1}  )$  \\
Multi-start    &  $O_p( d^3\ln d \cdot \ln \varepsilon^{-1} )$  &  $\times \ O_p( \varepsilon^{-1} )$       \\
\hline\hline\\
\end{tabular}
}\end{center}
\end{table}



\section{Applications}\label{Sec:Appl}

In this section we verify that the CLT conditions and the analysis apply
to a variety of statistical problems. In particular, we focus on
the MCMC estimator (\ref{qBest}) as an alternative to $M$- and $Z$-estimators.
Here our goal is to derive the high-level conditions C1-C3 from appropriate
primitive conditions, and thus show the efficient computational complexity of the MCMC estimator.

\subsection{M-Estimation} We present two examples in M-estimation.
We begin with the canonical log-concave cases
within the exponential family. Then we drop the concavity and
smoothness assumptions to illustrate the full applicability of the
approach developed in this paper.

\subsubsection{Exponential Family} Exponential families play a very important role in statistical
estimation, cf. Lehmann and Casella \cite{lehmann}, especially in
high-dimensional contexts; see Portnoy \cite{Portnoy}, Ghosal
\cite{G2000}, and Stone et al. \cite{Stone}. For example, the
high-dimensional situations arise in modern data sets in
technometric and econometric applications. Moreover, exponential
familes have excellent approximation properties and are useful for
approximation of densities that are not necessarily of the
exponential form; see Stone et al. \cite{Stone}.

We base our discussion on the asymptotic analysis of Ghosal
\cite{G2000}.  In order to simplify the exposition, we invoke the more
canonical conditions similar to those given in Portnoy
\cite{Portnoy}.  Moreover, we assume that these conditions, numbered as E.1 to E.4,
hold uniformly in the sample size $n$.

\begin{itemize}
\item[\textbf{E.1}] Let $X_1,\ldots,X_n$ be iid observations from
a $d$-dimensional canonical exponential family with density

$$
h(x;\theta) = \exp\( x'\theta - \psi(\theta)\),
$$
where $\theta \in \Theta$ is an open subset of $\RR^d$, and $d \to
\infty$ as $n \to \infty$.  Fix a sequence of parameter points
$\theta_0\in \Theta$.  Set $\mu = \psi'(\theta_0)$ and $J=
\psi''(\theta_0)$, the mean and covariance of the observations,
respectively.  Following Portnoy \cite{Portnoy}, we implicitly
re-parameterize the problem, so that the Fisher information matrix
$J=I$.
\end{itemize}

For a given prior $\pi$ on $\Theta$, the posterior density of
$\theta$ over $\Theta$ conditioned on the data takes the form
$$
\pi_n(\theta) \propto \pi(\theta) \cdot \prod_{i=1}^n
h(X_i;\theta) = \pi(\theta) \cdot \exp\( n \bar X'\theta -
n\psi(\theta)\),
$$where $\bar{X} = \sum_{i=1}^n X_i/n$ is the empirical mean of the data.

We associate every point $\theta$ in the parameter
space $\Theta$ with a local parameter $\lambda \in \Lambda= \sqrt{n}(\Theta - \theta) -s$, where $$
\lambda = \sqrt{n}(\theta - \theta_0) - s,
$$
and $s =\sqrt{n}(\bar x - \mu)$ is a first order approximation
to the normalized maximum likelihood/extremum estimate. By design,
we have that $E [s]= 0$ and $E \left[s s'\right] = I_d$. Moreover,
by Chebyshev's inequality, the norm of $s$ can be bounded in
probability, $\|s\| = O_p(\sqrt{d})$. Finally, the posterior
density of $\lambda$ over $\Lambda = \sqrt{n}(\Theta - \theta_0)
-s$ is given by $ f(\lambda) = \frac{\ell(\lambda)}{\int_{\Lambda}
\ell(\lambda) d \lambda}, $ where {\small
 \begin{align}\begin{split}
 \ell(\lambda) & = \exp\( \bar X'\sqrt{n}\lambda - n
 \psi\left(\theta_0 + \frac{\lambda +s}{\sqrt{n}}\right) +n
\psi\left(\theta_0 + \frac{s}{\sqrt{n}}\right) \) \\
& \times \pi\left(\theta_0 +
\frac{\lambda+s}{\sqrt{n}}\right)/\pi
\left (\theta_0 +
\frac{s}{\sqrt{n}} \right).
\end{split}\end{align}}
We impose the following regularity conditions, following Ghosal
\cite{G2000} and Portnoy \cite{Portnoy}:
\begin{itemize}
\item[\textbf{E.2}] Consider the following quantities associated
with higher moments in a neighborhood of the true parameter
$\theta_0$, uniformly in $n$ : \bs B_{1n}(c):= \sup_{\theta,\eta}\{ E_{\theta}|\eta'(x_i -
\mu)|^3: \eta \in S^d,  \|
\theta - \theta_0\|^2 \leq cd/n\}, \\
B_{2n}(c):= \sup_{\theta,\eta}\{ E_{\theta}|\eta'(x_i - \mu)|^4: \eta \in
S^d, \| \theta - \theta_0\|^2 \leq cd/n\}.
\end{split}\end{align}
where $S^d= \{\eta \in \RR^d: \|\eta\|=1\}$. There are $p>0$ and $c_0>0$ such that
$
B_{1n}(c)< c_0 + c^p \text{ and } B_{2n}(c) < c_0 + c^p
$
for all $c>0$ and all $n$.

\item[\textbf{E.3}] The prior density $\pi$ is proper and
satisfies a positivity requirement at the true parameter
$$
 \sup_{\theta \in
\Theta} \ln \[\pi(\theta)/\pi(\theta_0)\] = O(d)
$$ where $\theta_0$ is the true parameter.
Moreover, the prior $\pi$ also satisfies the following local
Lipschitz condition
$$
| \ln \pi(\theta) - \ln \pi(\theta_0) | \leq V(c) \sqrt{d} \|
\theta - \theta_0\|
$$
for all $\theta$ such that $\|\theta - \theta_0\|^2 \leq cd
 /n$, and some $V(c)$ such that $V(c) < c_0+c^p$, with the latter holding for all $c>0$. \\

\item[\textbf{E.4}] The parameter dimension $d$ grows at the rate such that $d^3 / n
\to 0.$

\end{itemize}

Condition E.2 strengthens an analogous condition of Ghosal
\cite{G2000}, and implies an analogous
assumption by Portnoy \cite{Portnoy}. Condition E.3 is
similar to the condition on the prior in Ghosal \cite{G2000}. For
further discussion of this condition, see \cite{BC}. Condition
E.4 states that the parameter dimension should not grow too
quickly relative to the sample size.

\begin{theorem}\label{Cond:Exp} Conditions E.1-E.4 imply conditions
C.1-C.3 with $\|K\| = C \sqrt{d}$ for some $C>1$.
\end{theorem}

\begin{remark}
Combining Theorems \ref{Thm:CLT} and
\ref{Cond:Exp}, we have the asymptotic normality
of the posterior,
$$
\int_{\Lambda} |f(\lambda) - \phi(\lambda)| d \lambda = o_p(1).
$$
Furthermore, we can apply Theorem \ref{Thm:Main}
to the posterior density $f$ to bound the
convergence time (number of steps) of the
Metropolis walk needed to obtain a draw from $f$
(with a fixed level of accuracy): The convergence
time is at most $$ O_p(d^2) $$ after the burn-in
period; together with the burn-in, the
convergence time is
$$O_p(d^3\ln d).$$ Finally, the integration bounds stated in the
previous section also apply to the posterior $f$.
\end{remark}

\subsubsection{Curved Exponential Family}
Next we consider the case of a $d$-dimen\-sio\-nal curved
exponential family. The curved family is general enough to allow for non-concavities and
even non-smoothness in the log-likelihood
function, which the canonical exponential family did not allow for.  We assume that the following conditions, numbered as NE.1 to NE.4,
hold uniformly in the sample size $n$,  in addition to the previous conditions E.1 to E.4.

\begin{itemize}

\item[\textbf{NE.1}] Let $X_1,\ldots,X_n$ be iid observations
from a $d$-dimensional curved exponential family with density
$$
h(x;\theta) = \exp\( x'\theta(\eta) - \psi(\theta(\eta))\).
$$
The parameter of interest is $\eta$, whose
true value $\eta_0$ lies in the interior of a convex compact set
$\Psi \subset \RR^{d_1}$.  The true value of $\theta$, induced by
$\eta_0$ is given by $\theta_0 = \theta(\eta_0)$. The mapping
$\eta \mapsto \theta(\eta)$ takes values from $\RR^{d_1}$ to
$\RR^d$ where $ c \cdot d \leq d_1 \leq d$, for some $c>0$. Finally, $d \to
\infty$ as $n \to \infty$.

\item[\textbf{NE.2}] True value $\eta_0$ is the unique solution to the
system $\theta(\eta) = \theta_0$, and we have that $\|\theta(\eta) -
\theta(\eta_0)\| \geq \epsilon_0 \|\eta - \eta_0\|$ for some
$\epsilon_0>0$ and all $\eta \in \Psi$.
\end{itemize}

Thus, the parameter $\theta$ corresponds to a high-dimensional
linear parametrization of the log-density, and $\eta$ describes
the lower-dimensional parametrization of the log-density.
There are many classical examples of curved exponential families;
see for example Efron \cite{efron}, Lehmann and Casella
\cite{lehmann}, and Bandorff-Nielsen \cite{bandorff}. An example
of the condition that puts a curved structure onto an exponential
family is a moment restriction of the type:
$$
\int m(x, \alpha) h(x,\theta) d x =0.
$$
This condition restricts $\theta$ to lie on a curve that can be
parameterized as $\{\theta(\eta), \eta \in \Psi\}$, where the
parameter $\eta=(\alpha, \beta)$ contains the component $\alpha$ as well as
other components $\beta$. In econometric applications,
moment restrictions often represent Euler equations that result from the
data $x$ being an outcome of an optimization by rational
decision-makers; see e.g. Hansen and Singleton
\cite{hansen:singleton}, Chamberlain \cite{chamberlain}, Imbens
\cite{imbens}, and Donald, Imbens and Newey
\cite{Donald-Imbens-Newey}. Thus, the curved exponential framework
is a fundamental complement of the exponential framework, at least
in certain fields of data analysis.

We require the following additional regularity conditions on the
mapping $\theta(\cdot)$:
\begin{itemize}
\item[\textbf{NE.3}] For every $\kappa$, and uniformly in $\gamma
\in B(0,\kappa\sqrt{d})$, there exists a linear operator
$G:\RR^{d_{1}}\to \RR^d$ such that $G'G$ has eigenvalues bounded
from above and away from zero, uniformly in $n$, and for every $n$
$$
\sqrt{n}\left( \theta(\eta_0 + \gamma/\sqrt{n})-\theta(\eta_0)
\right) = r_{1n} + (I_d+R_{2n})G\gamma,
$$ where $\|r_{1n}\|\leq \delta_{1n}$ and $\|R_{2n}\|\leq
\delta_{2n}$ and
$
\delta_{1n}\sqrt{d}  \to 0 \ \ \mbox{and}\ \ \delta_{2n} d \to 0.
$
\end{itemize}

Thus the mapping $\eta \mapsto \theta(\eta)$ is allowed to be
nonlinear and discontinuous. For example, the additional condition
of $\delta_{1n}=0$ implies the continuity of the mapping in a
neighborhood of $\eta_0$. More generally, condition NE.3 does
impose that the map admits an approximate linearization in the
neighborhood of $\eta_0$, whose quality is controlled by the errors
$\delta_{1n}$ and $\delta_{2n}$. An example of a kind of map
allowed in this framework is given in Figure \ref{Fig:2}.

\begin{figure}
\begin{center}
\includegraphics[width=0.6\textwidth]{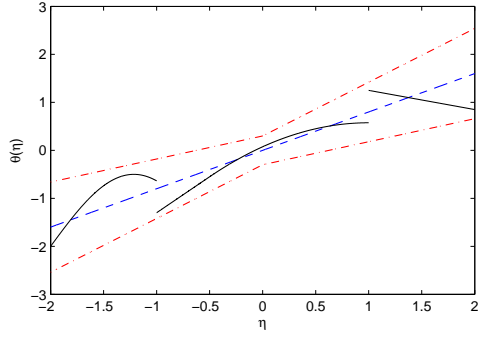}
\end{center}
\caption{This figure illustrates the mapping $\theta(\cdot)$. The
(discontinuous) solid line is the mapping while the dash line
represents the linear map induced by $G$. The dash-dot line
represents the deviation band controlled by $r_{1n}$ and
$R_{2n}$.}\label{Fig:2}
\end{figure}

Given a prior $\pi$ on $\Theta$, the posterior of $\eta$
given the data is denoted by
$$
\pi_n(\eta) \propto \pi(\theta(\eta)) \cdot \prod_{i=1}^n
h(X_i;\eta) = \pi(\theta(\eta)) \cdot \exp\( n \bar X'\theta(\eta)
- n\psi(\theta(\eta))\).
$$
In this framework, we also define the local parameters to describe
contiguous deviations from the true parameter as
$$
\gamma = \sqrt{n}(\eta - \eta_0) - s,  \ \ s =
(G'G)^{-1}G'\sqrt{n}(\bar x - \mu),
$$
where $s$ is a first order approximation to the normalized maximum
likelihood/extremum estimate. Further, we have that
$E [s]= 0$, $E [s s'] = (G'G)^{-1}$, and $\|s\| = O_p(\sqrt{d})$.
The posterior density of $\gamma$ over $\Gamma$, where $\Gamma =
\sqrt{n}(\Psi-\eta_0)-s$, is $ f(\gamma) =
\frac{\ell(\gamma)}{\int_{\Gamma} \ell(\gamma) d \gamma}, $ where
\bsnumber \ell(\gamma) & = \exp\left( n \bar X' \left (\theta\left (\eta_0 +
\frac{\gamma+s}{\sqrt{n}}\right ) - \theta\left (\eta_0 +
\frac{s}{\sqrt{n}}\right)\right )\right ) \\
& \times \exp\left( -n\psi\left (\theta\left (\eta_0 +
\frac{\gamma+s}{\sqrt{n}}\right )
\right) + n\psi\left(\theta\left (\eta_0 +
\frac{s}{\sqrt{n}}\right) \right) \right ) \\
&\times \pi\left ( \theta\left (\eta_0 +
\frac{\gamma+s}{\sqrt{n}}\right )\right)/\pi\left (\theta\left (\eta_0 +
\frac{s}{\sqrt{n}}\right) \right).
\end{split}\end{align}
The condition on the prior is the following:
\begin{itemize}
\item[] \textbf{NE.4} The prior $\pi(\eta) \propto
\pi(\theta(\eta))$, where $\pi(\theta)$ satisfies condition E.3.\\
\end{itemize}
\begin{theorem}\label{Cond:CExp} Conditions E.1-E.4 and NE.1-NE.4  imply conditions C.1-C.3 with $\|K\| = C \sqrt{d/\lambda_{min}}$ for some $C>1$, where $\lambda_{\min}$ is the minimal eigenvalue of $J=G'G$.
\end{theorem}

\begin{remark} Theorems \ref{Thm:CLT} and
\ref{Cond:CExp} imply the asymptotic normality of
the posterior,
$$
\int_{\Gamma} |f(\gamma) - \phi(\gamma)| d \gamma = o_p(1),
$$
where
$$
\phi(\gamma) =  \frac{1}{ (2\pi)^{d/2} \det{((G'G)^{-1})}^{1/2}}
\exp \(- \frac 12 \gamma' (G'G) \gamma\).
$$
Theorem \ref{Thm:Main} implies further that the main results of
the paper on the polynomial time sampling and integration apply to
this curved exponential family. \end{remark}

\subsection{Z-estimation} Next we turn to the $Z$-estimation problem, where our basic setup closely follows the setup
in e.g. He and Shao \cite{HS00}.  We make the following assumption that characterizes
the setting. As in the rest of the paper, the dimension of the parameter space $d$ and other quantities will depend on the sample size $n$.

\begin{itemize}
\item[\textbf{ZE.0}]   The data $X_1,...,X_n$ are i.i.d, and there exists a vector-valued moment function  $m:\mathcal{X} \times \RR^d \to \RR^{d_1}$
such that $$E[m(X,\theta)] = 0 \ \ \mbox{at the true parameter  } \theta = \theta_0 \in \Theta_n \subset B(\theta_0, T_n) \subset \RR^{d}.$$
Both the dimension of the moment function $d_1$ and the dimension
of the parameter $d$ grow with the sample size $n$, and we restrict
that $c d_1 \leq d \leq d_1$ for some constant $c$. The parameter space $\Theta_n$ is an open convex set contained in the ball $B(\theta_0, T_n)$ of radius
$T_n$, where the radius $T_n$ can grow with the sample size $n$.
 \end{itemize}
  The normalized empirical moment function takes the form

$$ S_n(\theta) = \frac{1}{\sqrt{n}}\sum_{i=1}^n m(X_i,\theta). $$
The $Z$-estimator for $\theta_0$ is defined as the minimizer of the norm $\|S_n(\theta)\|$.   However, in many applications of interests,
 the lack of continuity or smoothness of the empirical moments $S_n(\theta)$
can pose serious computational challenges to obtaining the minimizer.
As argued in the introduction, in such cases the MCMC
methodology could be particularly appealing for obtaining the quasi-posterior means and medians
as computationally tractable alternatives to the Z-estimator based on minimization.

We then make the following variance and smoothness assumptions on the moment functions in addition to the basic condition \textbf{ZE.0}:

\begin{itemize}

\item[\textbf{ZE}.1]  Let $S^{d_1}= \{\eta \in \RR^{d_1}: \|\eta \|= 1\}$ denote the unit sphere. The
variance of the moment function is bounded, namely $\sup_{\eta \in S^{d_1}}  \Ep[ (\eta' m(X,\theta_0))^2]=O(1)$.
The moment functions have the following continuity property:  $\sup_{\eta \in S^{d_1}} ( \Ep [  (\eta'(m(X,\theta) - m(X,\theta_0)))^2])^{1/2} \leq  O(1) \cdot \|\theta - \theta_0\|^{\alpha}$,
uniformly in $\theta \in \Theta_n$,  where
$\alpha \in (0,1]$ and is bounded away from zero, uniformly in $n$.  Moreover,  the family of functions $\mathcal{F} = \{
\eta' (m(X,\theta)- m(X,\theta_0)) \ : \theta \in \Theta_n \subset \RR^d, \eta \in S^{d_1}  \}$ is not very complex, namely
the uniform covering entropy of $\mathcal{F}$ is of the same order as the uniform covering entropy of a Vapnik-Chervonenkis  (VC) class
of functions with VC dimension of order $O(d)$, and $\mathcal{F}$ has an envelope $F$ a.s. bounded by $M=O(\sqrt{d})$.
\end{itemize}

The smoothness assumption covers moment function both in  the smooth case, where $\alpha=1$,
 and the non-smooth case, where $\alpha < 1$.  For example, in the classical mean regression problem, we have the smooth case $\alpha=1$
 and in the quantile regression problems  mentioned in the introduction, we have a non-smooth case, with
 $\alpha = 1/2$.  The condition on the function class $\mathcal{F}$ is standard in statistical estimation and, in particular, holds for $\mathcal{F}$ formed as VC classes
 or certain stable transformations of VC classes (see van der Vaart and Wellner \cite{VW}).
 We use the entropy in conjunction with the maximal
 inequalities similar to those developed in He and Shao \cite{HS00}.  The condition on the envelope is standard,
 but it can be replaced by an alternative condition on $\sup_{f \in \mathcal{F}} n^{-1} \sum_{i=1}^n f^4$, see e.g. He and Shao \cite{HS00},
 which can weaken the assumptions on the envelope.

Next we make the following additional smoothness  and identification assumptions uniformly in the sample size $n$.

\begin{itemize}

\item[\textbf{ZE}.2] The mapping $\theta \mapsto \Ep[ m(X, \theta) ]$ is continuously twice
differentiable with $\|\sup_{\eta \in S^{d_1}}\nabla^2_{\theta} \Ep[ m(X, \theta) ][\eta, \eta]\|$ bounded by $O(\sqrt{d})$
uniformly in $\theta$, uniformly in $n$.   The eigenvalues of  $A'A$, where $A=\nabla \Ep[ m(X,\theta_0) ]$
is the Jacobian matrix,  are bounded above and away from zero uniformly in $n$. Finally, there exist positive numbers $\mu$
and $\delta$ such that uniformly in $n$, the following identification condition holds \begin{equation}\label{Eq:ZE2} \left\|
\Ep\[ m(X,\theta)
\]\right\| \geq  \left( \sqrt{\mu} \|\theta-\theta_0\| \
\wedge \ \delta \ \right).\end{equation}
\end{itemize}
This condition requires the population moments $\Ep[ m(X,\theta)]$ to be approximately
linear in the parameter $\theta$ near the true parameter value $\theta_0$, and also insures identifiability of the true parameter value $\theta_0$.

Finally, we  impose the following restrictions on the parameter dimension $d$ and the radius of the parameter space $T_n$.

\begin{itemize}
\item[\textbf{ZE}.3] The following condition holds: (a)
$d^4\log^2 n/ n \to 0$, (b) $d^{2+\alpha} \log n/ n^{\alpha} \to 0$, and (c)
 $ d T_n^{2\alpha} \log n / n \to 0$.
\end{itemize}
These conditions are reasonable. Indeed, if we set $\alpha=1$ and use radius $T_n = O(d \log n)$ for parameter space, then we require only that $d^4/ n \to  0$, ignoring logs, which is only slightly stronger
than the condition $d^3/n \to 0$ needed in the exponential family case.   In the latter case, the information on higher order moments lead to the weaker requirement. Also, an important difference here is that we are using the flat
prior in the Z-estimation framework, and this necessitates us to restrict the radius of parameter space by $T_n$.  Note that even though the bounded radius $T_n= O(1)$ is already plausible for many applications,
we can allow for the radius to grow, for example, $T_n = O(d \log n)$ when $\alpha =1$.

In order to state the formal results concerning the quasi-posterior,  let us define the quasi-posterior
and related quantities. First, we define the criterion function as
$ Q_n(\theta) = -\|
S_n(\theta)\|^2,$
and treat it as a replacement for the log-likelihood. We will use a flat prior over the parameter space $\Theta$, so that the quasi-posterior density of
$\theta$ over $\Theta$ takes the form
$$
\pi_n(\theta)  = \frac{\exp\{Q_n(\theta)\}}{\int_\Theta \exp\{ Q_n(\theta') \} d \theta'}.
$$
We associate every point $\theta$ in the parameter
space $\Theta$ with a local parameter $\lambda \in \Lambda = \sqrt{n} (\Theta - \theta_0)-s$, where $ \lambda = \sqrt{n}(\theta - \theta_0) - s,$ and $s = - (A'A)^{-1} A' S_n(\theta_0)$  is a first order approximation
to extremum estimate. We have that $E[m(X,\theta_0) m(X,\theta_0)']$ is bounded
in the spectral norm, and $(A'A)^{-1} A'$ has a bounded norm, so that the norm of $s$ can be bounded in
probability, $\|s\| = O_p(\sqrt{d})$, by the Chebyshev inequality. Finally, the quasi-posterior
density of $\lambda$ over $\Lambda = \sqrt{n}(\Theta - \theta_0)
-s$ is given by $$ f(\lambda) = \ell(\lambda)/ \int_{\Lambda}
\ell(\lambda') d \lambda', $$ where
$$ \ell(\lambda) =  \exp( Q_n(\theta_0+ (\lambda
+s)/\sqrt{n}) - Q_n(\theta_0 + s/\sqrt{n} ) ) .
$$

\begin{theorem}\label{Theorem:Zest} Conditions ZE.0-ZE.3 imply conditions C.1-C.3 with $\|K\| = C \sqrt{d/\lambda_{min}}$ for $C>1$, where $\lambda_{\min}$ is the minimal eigenvalue of $J=2 A'A$.
\end{theorem}

\begin{remark}
Theorems \ref{Thm:CLT} and
\ref{Theorem:Zest} imply the asymptotic normality of the quasi-posterior,
$$
\int_{\Lambda} |f(\lambda) - \phi(\lambda)| d \lambda = o_p(1),
$$
where
$$
\phi(\lambda) =  \frac{1}{ (2\pi)^{d/2} \det{ J}^{1/2}}
\exp \(- \frac{1}{2}\lambda' J \lambda\).
$$
Theorem \ref{Thm:Main} implies further that the main results of
the paper on the polynomial time sampling and integration apply to
the quasi-posterior density formulated for the Z-estimation framework. \end{remark}

\section{Conclusion}\label{Sec:Improv}

In this paper we study the implications of the
statistical large sample theory for computational
complexity of Bayesian and quasi-Bayesian
estimation carried out using a canonical
Metropolis random walk. Our analysis permits the
parameter dimension of the problem to grow to
infinity and allows the underlying log-likelihood
or extremum criterion function to be
discontinuous and/or non-concave. We establish
polynomial complexity by exploiting a central
limit theorem framework which provides the structural
restriction on the problem, namely, that the posterior
or quasi-posterior density approaches a normal
density in large samples.

We focused the analysis on (general) Metropolis random
walks and provided specific bounds for a canonical Gaussian random walk.
Although it is widely used for its simplicity, this canonical
random walk is not the most sophisticated algorithm available.
Thus, in principle further improvements could be obtained by
considering different kinds of algorithms, for example, the
Langevin diffusion \cite{RGG,Stramer-Tweedie,RT,Atchade}. (Of course, the
algorithm requires a smooth gradient of the log-likelihood
function, which rules out the nonsmooth and discontinuous cases
emphasized here.) Another important research direction, as
suggested by a referee, could be to develop sampling and
integration algorithms that most effectively exploit the proximity
of the posterior to the normal distribution.

\section*{Acknowledgement}

We would like to thank Ron Gallant, Lars Hansen,
Jerry Hausman, James Heckman, Roger Koenker,
Steve Portnoy, Nick Polson, Santosh Vempala and
participants of seminars at the University of
Chicago, MIT, Duke, and the INFORMS Conference,
for useful comments. We also thank two referees, an associate
editor of the journal, Moshe Cohen, Raymond Guiteras, Alp Simsek,
Paul Schrimpf, and Theophane Weber for
thorough readings of the paper and their valuable
comments that have considerably improved this
paper.

\begin{appendix}

\section{Proofs of Other Results}\label{App:Proofs}

\Aproof{Theorem}{Thm:CLT} From C.1 it follows that
$$
\begin{array}{rcl}
\displaystyle \int_{\Lambda} |f(\lambda) - \phi(\lambda)| d
\lambda & \leq &\displaystyle  \int_{K} |f(\lambda) -
\phi(\lambda)| d \lambda + \int_{K^{c}}\left( f(\lambda) + \phi(\lambda) \right)  d\lambda \\
&=&\displaystyle  \int_{K} |f(\lambda) - \phi(\lambda)| d \lambda + o_p(1)\\
\end{array}
$$

Now, denote $\displaystyle C_n = \frac{(2\pi)^{d/2} \det{(J^{-1})}^{1/2}}{\int_K \ell(\omega) d\omega} $ and write %
$$
\begin{array}{rcl}
\displaystyle \int_{K} \left|\frac{f(\lambda)}{\phi(\lambda)} -
1\right| \phi(\lambda) d \lambda & = &\displaystyle  \int_{K}
\left|C_n \cdot
\exp\left(\ln \ell(\lambda) - \left(-\frac{1}{2} \lambda'J\lambda\right) \right) - 1 \right| \phi(\lambda) d \lambda \\
\end{array}
$$
Combining the expansion in C.2 with conditions imposed in C.3,
$$
\begin{array}{rcl}
\displaystyle \int_{\Lambda}
\left|\frac{f(\lambda)}{\phi(\lambda)} - 1\right| \phi(\lambda) d
\lambda & \leq &  \int_{K} \left|C_n \cdot \exp\left( \epsilon_1 +
\epsilon_2 \lambda'J\lambda \right) - 1 \right|
\phi(\lambda) d \lambda \\
& & + \int_{K} \left|C_n \cdot
\exp\left( -\epsilon_1 - \epsilon_2 \lambda'J\lambda \right) - 1 \right| \phi(\lambda) d \lambda\\
&\leq & \displaystyle 2\int_K\left|C_n \cdot e^{o_{p}(1)} - 1 \right| \phi(\lambda) d \lambda\\
& \leq &  2|C_n  e^{o_p(1)} -1 |
\end{array}
$$
The proof then follows by showing that $C_n \to_p 1$. Using condition C.1 on the set $K=B(0, \|K\|)$ and C.2,
$$
\begin{array}{rcl}
\displaystyle \frac{1}{C_n} & \geq & \frac{\displaystyle  \int_{K} \ell (\lambda) d\lambda}{\displaystyle (1 + o(1) )\int_{K} e^{-\frac{1}{2}\lambda'J\lambda} d\lambda} %
\geq  \displaystyle \frac{\displaystyle  \int_{K}
e^{-\frac{1}{2}\lambda'J\lambda} e^{-\epsilon_1-
\frac{\epsilon_2}{2} (\lambda'J\lambda)} d\lambda}
{\displaystyle (1 + o(1) )\int_{K} e^{-\frac{1}{2}\lambda'J\lambda}  d\lambda} \\
\\
&  =  & \displaystyle
\frac{e^{-\epsilon_1}}{(1 + o(1) )}\sqrt{\frac{\det(J)}{\det(J+\epsilon_2 J)}} \frac{\displaystyle  \int_{K}  \frac{e^{-\frac{1}{2}\lambda'(J+\epsilon_2 J)\lambda}}{(2\pi)^{d/2}\det((J+\epsilon_2J)^{-1})^{1/2}} d \lambda} {\displaystyle  \int_{K} \frac{e^{-\frac{1}{2}\lambda'J\lambda}}{(2\pi)^{d/2}\det(J^{-1})^{1/2}} d\lambda}. %
\end{array}
$$

\noindent Since $\epsilon_2 < 1/2$, we can define $W \sim
N(0,(1+\epsilon_2)^{-1}J^{-1})$ and $V \sim N(0,J^{-1})$ and
rewrite our bound as $$
\begin{array}{rcl}
\displaystyle \frac{1}{C_n}
&\geq & \displaystyle  \frac{e^{-\epsilon_1}}{(1 + o(1) )}\left( \frac{1}{1+\epsilon_2}\right)^{d/2} \frac{ P ( \|W\| \leq \|K\| ) }{P ( \|V\| \leq \|K\| )} \\ %
&\geq & \displaystyle  \frac{e^{-\epsilon_1}}{(1 + o(1) )}\left( \frac{1}{1+\epsilon_2}\right)^{d/2}  \\ %
\end{array}
$$ \noindent where the last inequality follows from $ P ( \|W\| \leq
\|K\|) \geq P( \|\sqrt{1+\epsilon_2} W\| \leq \|K\| ) = P ( \|V\| \leq \|K\|
)$. Likewise, $$
\begin{array}{rcl}
\displaystyle \frac{1}{C_n} \leq \frac{\displaystyle  \int_{K} \ell (\lambda) d\lambda}{\displaystyle \int_{K} e^{-\frac{1}{2}\lambda'J\lambda} d\lambda}
&\leq & \displaystyle  e^{\epsilon_1} \left( \frac{1}{1-\epsilon_2}\right)^{d/2}  \\ %
\end{array}
$$
Therefore $C_n \to_p 1$ since $\epsilon_1 \to_p 0$ and $\epsilon_2 \cdot
d \to_p 0$ (cf. Comment \ref{comment: on conditions}). $\qed$

\Aproof{Lemma}{Lemma1}  The result follows
immediately from equations (\ref{Def:g})-(\ref{Def:m}).  \qed

\Aproof{Lemma}{betaISO} Let $M := \beta \frac{2t
e^{-t^2/4}}{\sqrt{\pi}}$. Take any measurable partition of $K = S_1
\cup S_2 \cup S_3$, with $d(S_1,S_2) \geq t$. It  suffices to prove that
$$ \int \left( M1_{S_{i}}(x) - 1_{S_{3}}(x) \right) f(x) dx < 0 ,\mbox{~for ~} i = 1 \text{ or } i= 2.$$
We will prove this by
contradiction.  Suppose that
$$ \int \left( M1_{S_{i}}(x) - 1_{S_{3}}(x) \right) f(x) dx > 0 ,\mbox{~for ~} i = 1 \text{ and } i=2.$$
\noindent We will use the Localization Lemma of Kannan, Lov\'asz,
and Simonovits \cite{KLS} in order to reduce a high-dimensional
integral to a low-dimensional integral.
\begin{lemma}[Localization Lemma]
Let $g$ and $h$ be two lower semi-continuous Lebesgue integrable
functions on $\RR^d$ such that
$$ \int_{\RR^d} g(x) dx > 0 \mbox{~~~and~~~} \int_{\RR^d} h(x) dx > 0. $$
\noindent Then there exist two points $a,b \in \RR^d$, and a
linear function $\tilde \gamma:[0,1] \to \RR_+$ such that
$$ \int_0^1 \tilde \gamma^{d-1}(t) g((1-t)a+tb)dt > 0 \mbox{~~~and~~~}  \int_0^1 \tilde \gamma^{d-1}(t) h((1-t)a+tb) dt > 0, $$%
\noindent where $([a,b],\tilde \gamma)$ is said to form a needle.
\end{lemma}
\begin{proof} See Kannan, Lov\'asz, and Simonovits \cite{KLS}. \end{proof}

By the Localization Lemma, there exists a needle $(a,b,\tilde
\gamma)$ such that
$$ \int_0^{1} \tilde \gamma^{d-1}(l) f((1-l)a+lb)\left( M 1_{S_{i}}((1-l)a+lb) - 1_{S_{3}}((1-l)a+lb) \right) du > 0 ,$$
\noindent for $i = 1,2$. Equivalently, using $\gamma(u)=\tilde
\gamma(u/\|b-a\|)$ and $v := (b-a)/\|b-a\|$ where $\|b-a\| \geq
t$, and rearranging we have for $i=1, 2$,
\begin{align}\begin{split}\label{LL:00}
M \int_0^{\|b-a\|} \gamma^{d-1}(u) & f(a+uv) 1_{S_{i}}(a+uv)du \\
  & > \int_0^{\|b-a\|} \gamma^{d-1}(u) 1_{S_{3}}(a+uv) f(a+uv)du.
\end{split}\end{align}

In order for the left hand side of (\ref{LL:00}) be positive for
$i=1$ and $i=2$, the line segment $[a,b]$ must contain points in
$S_1$ and $S_2$. Since $d(S_1,S_2)\geq t$, we have that $S_3 \cap
[a,b]$ contains an interval $[w,w+t]$ whose length is at least $t$. Thus, we can partition the line segment $[a,b]$ into $[0, w) \cup [w, w+t] \cup (w+t, \|b- a\|]$. We will prove that for every $w \in \RR $ such that $0 \leq w \leq w+t \leq \|b-a\|${\small
\begin{equation}\label{MainEq:01}
\begin{array}{rc}
\displaystyle  \int_w^{w+t} \gamma^{d-1}(u) f(a+uv)du &
\displaystyle\geq M \min\left\{
\int_{0}^w \gamma^{d-1}(u) f(a+uv)du, \right. \\
&\displaystyle \ \left. \int_{w+t}^{\|b-a\|} \gamma^{d-1}(u)
f(a+uv)du \right\}
\end{array}
\end{equation}} which contradicts the relation (\ref{LL:00}) and
proves the lemma.

 First, note that $f(a+uv) = e^{-\|a+uv\|^2}m(a+uv) =
e^{-u^2 + r_1 u + r_0} m(a+uv)$ where $r_1 := 2 a'v$ and $r_0 :=
-\|a\|^2$. Next, recall that $m(a+uv)\gamma^{d-1}(u)$ is still a
unidimensional log-$\beta$-concave function on $u$. By Lemma
\ref{STRbr} presented in Appendix B, there exists a unidimensional
logconcave function $\hat m$ such that $ \beta \hat m(u) \leq
m(a+uv)\gamma^{d-1}(u) \leq \hat m(u)$ for every $u$. Moreover,
there exists numbers $s_0$ and $s_1$ such that $\hat m (w) = s_0
e^{s_1 w}$ and $\hat m (w+t) = s_0 e^{s_1(w+t)}$. Due to the
log-concavity of $\hat m$, this implies that
$$ \hat m(u) \geq s_0 e^{s_1 u} \ \ \mbox{for} \ u \in (w,w+t) \ \
\mbox{and} \ \ \hat m(u) \leq s_0 e^{s_1 u} \ \
\mbox{otherwise.}$$

Thus, if we replace $m(a+uv)\gamma^{d-1}(u)$ by $s_0 e^{s_1u}$ on
the right hand side of (\ref{MainEq:01}) and replace
$m(a+uv)\gamma^{d-1}(u)$ by $\beta s_0 e^{s_1u} $ on the left hand
side of (\ref{MainEq:01}), and define $\hat r_1 = r_1 + s_1$
and $\hat r_0:= r_0 + \ln s_0$, we obtain the relation  $$ \beta
\int_w^{w+t} e^{-u^2 + \hat r_1 u +\hat r_0} du \geq M \min\left\{
\int_{0}^w e^{-u^2 + \hat r_1 u +\hat r_0} du, \
\int_{w+t}^{\|b-a\|} e^{-u^2 + \hat r_1 u +\hat r_0} du \right\}.
$$ This relation is stronger than (\ref{MainEq:01}) and thus implies (\ref{MainEq:01}). This relation is equivalent to
 \begin{equation}\label{MainEq:04} \begin{array}{rc} \beta
\displaystyle \int_w^{w+t} e^{-(u - \frac{\hat r_1}{2})^2 + \hat
r_0 + \frac{\hat r_1^2}{4}} du &\displaystyle \geq M \min\left\{
\int_{0}^w e^{-(u - \frac{\hat r_1}{2})^2 + \hat r_0 + \frac{\hat
r_1^2}{4}} du,
\right. \\
& \displaystyle\left. \ \int_{w+t}^{\|b-a\|} e^{-(u - \frac{\hat
r_1}{2})^2 + \hat r_0 + \frac{\hat r_1^2}{4}} du \right\}.
\end{array}\end{equation} Now, cancel the term $e^{\hat r_0 + \hat
r_1^2/4}$ on both sides and, since we want the inequality
(\ref{MainEq:04}) holding for any $w$,  (\ref{MainEq:04}) is
implied by
\begin{equation}\label{MainEq:05}
\int_w^{w+t}  e^{-u^2}  du \geq \frac{2te^{-t^2/4}}{\sqrt{\pi}}
\min\left\{ \int_{-\infty}^w  e^{-u^2} du, \ \int_{w+t}^{\infty}
e^{-u^2} du \right\}
\end{equation}
holding for any $w$.  This inequality is Lemma 2.2 in Kannan and
Li \cite{KLi}.$\qed$


\Aproof{Corollary}{Corollary:Iso} Consider the change of variables
$\tilde x = \frac{J^{1/2}x}{\sqrt{2}}$ and $\tilde S = \frac{J^{1/2}S}{\sqrt{2}}$. Then, in $\tilde x$
coordinates, $f(\tilde x) = e^{\tilde x'\tilde
x}m(\sqrt{2}J^{-1/2}\tilde x)$ satisfies the assumption of Lemma
\ref{betaISO} and $d(\tilde S_1, \tilde S_2) \geq \frac{t\sqrt{\lambda_{min}}}{\sqrt{2}}$. The
result follows by applying Lemma \ref{betaISO} with $\tilde x$
coordinates. $\qed$

  \Aproof{Lemma}{GeoLips} The result is immediate from the stated assumptions. $\qed$

\Aproof{Theorem}{Thm:Main} See section \ref{Sec:ProofMain}. $\qed$

 \Aproof{Lemma}{GeoProb} Define $K :=
B(0,R)$, so that $R$ is the radius of $K$; also let $r:=
4\sqrt{d}\sigma$ (where $\sigma^2 \leq \frac{1}{16dL^2}$), and let
$q(x|u)$ denote the normal density function centered at $u$ with
covariance matrix $\sigma^2I$. We use the following notation: $B_u
= B(u,r)$, $B_v = B(v,r)$, and $A_{u,v} = B_u \cap B_v \cap K$. By
definition of $r$, we have that $\int_{B_u} q(x|u) dx = \int_{B_v}
q(x|v) dx \geq 1- P\{|U| \geq 4\} > 1- 1/10^4$, where $U \sim N(0,1)$.

Define the direction $w = (v-u)/\|v-u\|$.  Let $H_1 = \{x \in B_u
\cap B_v : w'(x-u) \ge \|v-u\|/2\}$, $H_2 = \{x \in B_u \cap B_v :
w'(x-u) \le \|v-u\|/2\}$.   Consider the one-step distributions
from $u$ and $v$.  We first observe that in view of Lemma \ref{Lemma1} and Lemma \ref{GeoLips} that
$\inf_{x\in B(y,r)} f(x)/f(y) \geq \beta e^{-Lr}$.
Then we have that
\begin{eqnarray*}
& & \|P_u-P_v\|_{TV} \le 1 - \int_{K}
\min\{dP_u,dP_v\} \leq 1 - \int_{A_{u,v}}
\min\{dP_u,dP_v\} \\
& & = 1 - \int_{A_{u,v}} \min
\left\{q(x|u)\min\left\{\frac{f(x)}{f(u)},1\right\},q(x|v)\min\left\{\frac{f(x)}{f(v)},1\right\}\right\}
dx \\
& & \le 1 - \beta e^{-Lr} \int_{A_{u,v}} \min
\left\{q(x|u),q(x|v)\right\}
dx \\
& & \le 1 - \beta e^{-Lr} \left( \int_{H_1 \cap K} q(x|u) dx +
\int_{H_2 \cap K} q(x|v) dx \right), \end{eqnarray*} where
$\|u-v\| < \sigma/8$.   Next we will bound from below the last sum
of integrals for an arbitrary $u \in K$.

We first bound the integrals over the possibly larger sets,
respectively $H_1$ and $H_2$.    Let $h$ denote the density
function of a univariate random variable distributed as
$N(0,\sigma^2)$.  It is easy to see that
$h(t)=\int_{w'(x-u)=t}q(x|u)dx$, i.e. $h$ is the marginal density
of $q(\cdot|u)$ along the direction $w$  up to a translation.   Let
$H_3 = \{x : - \|u-v\|/2 < w'(x-u) < \|v-u\|/2\}$.  Note that $B_u
\subset H_1 \cup \left(H_2 - \|u-v\|w\right)\cup H_3$ where the
union is disjoint.  Armed with these observations, we have {\small
\begin{eqnarray} \int_{H_1} q(x|u) dx + \int_{H_2} q(x|v) dx &=&
\int_{H_1}
q(x|u) dx + \int_{H_2 - \|u-v\|w} q(x|u)dx \notag \\
&\ge& \int_{B_u} q(x|u)dx - \int_{H_3} q(x|u)dx \notag \\
&=& \int_{B_u} q(x|u)dx - \int_{-\|u-v\|/2}^{\|u-v\|/2}h(t)dt \notag \\
&\ge& 1 - \frac{1}{10^4}-
\int_{-\|u-v\|/2}^{\|u-v\|/2}\frac{e^{-t^2/2\sigma^2}}{\sqrt{2
\pi}
\sigma}dt \notag \\
&\ge& 1 - \frac{1}{10^4} - \|u-v\|\frac{1}{\sqrt{2 \pi}\sigma} \notag \\
&\ge & 1 - \frac{1}{10^4} - \frac{1}{8\sqrt{2 \pi}} \ge
\frac{9}{10}, \label{noKineq}
\end{eqnarray}}
where we used that $\|u-v\|<\sigma/8$ by the hypothesis of the
lemma.

In order to take the support $K$ into account, we can assume that
$u,v \in \partial K$, i.e. $\|u\|=\|v\|=R$ (otherwise the integral
will be larger).  Let $z = (v+u)/2$ and define the half space $H_z
= \{x: z'x \le z'z\}$ whose boundary passes through $u$ and $v$
(Using $\|u\|=\|v\|=R$ it follows that $z'v=z'u=z'z/2$).

 By the symmetry of the normal density, we have
$$ \int_{H_1 \cap H_z} q(x|u) dx = \frac{1}{2}\int_{H_1} q(x|u) dx.
$$
Although $H_1 \cap H_z$ does not lie in $K$ in general, simple
arithmetic shows that $H_1 \cap \left(H_z -
\frac{r^2z}{R\|z\|}\right) \subseteq K$.\footnote{Indeed, take $y
\in H_1 \cap \left( H_z - \frac{r^2}{R} \frac{z}{\|z\|} \right) $.
We can write $y =
\frac{z}{\|z\|}\left(\frac{y'z}{\|z\|}\right)+s$, where $\|s\|\leq
r$ (since $\left\|y -
\frac{z}{\|z\|}\left(\frac{y'z}{\|z\|}\right)\right\| \leq \|y-z\|
= \|y - \frac{u+v}{2}\| \leq \frac{1}{2}\|y -u \| + \frac{1}{2}\|y
- v\| \leq r$) and $s$ is also orthogonal to $z$. Since $y \in
\left( H_z - \frac{r^2}{R} \frac{z}{\|z\|} \right)$, we have
$\frac{y'z}{\|z\|} \leq \frac{z'z}{\|z\|} - \frac{r^2}{R} = \|z\|
- \frac{r^2}{R} \leq R - \frac{r^2}{R}$. Therefore, $\| y \| =
\sqrt{ \left(\frac{y'z}{\|z\|}\right)^2 + \|s\|^2} \leq \sqrt{ (R
- \frac{r^2}{R})^2 +r^2}=\sqrt{R^2 - r^2(1- \frac{r^2}{R^2})} \leq R.$}

Using that $\int_{H_z \setminus (H_z - \frac{r^2z}{R\|z\|})}
q(x|u) = \int_0^{r^2/R}h(t)dt$, we have

{\small \begin{eqnarray*} \int_{H_1 \cap K} q(x|u)dx &\ge &
\int_{H_1 \cap \left(H_z - \frac{r^2z}{R\|z\|}\right)} q(x|u)dx
\ge \int_{H_1 \cap
H_z}q(x|u)dx - \int_0^{r^2/R}h(t)dt \\
&\ge& \frac{1}{2}\int_{H_1}q(x|u)dx -
\int_0^{r^2/R}\frac{e^{-t^2/2\sigma^2}}{\sqrt{2\pi}\sigma}dt \\
&\ge& \frac{1}{2}\int_{H_1}q(x|u)dx - 4\sqrt{d}\sigma
\frac{1}{30\sqrt{d}}\frac{1}{\sqrt{2\pi}\sigma}, \end{eqnarray*}}
where we used that $\frac{r}{R} < \frac{1}{30 \sqrt{d}}$ since $r
= 4\sqrt{d}\sigma$ and $\frac{\sigma}{R} < \frac{1}{120d}$.

 By symmetry, the same inequality holds when $u$ and $H_1$
are replaced by $v$ and $H_2$ respectively.  Adding these
inequalities and using (\ref{noKineq}), we have
\begin{equation}\label{bound over H1 and H2} \left(
\int_{H_1 \cap K} q(x|u) dx + \int_{H_2 \cap K} q(x|v) dx \right)
\ge \frac{9}{20} - \frac{4}{15 \sqrt{2 \pi}} \ge 1/3.
\end{equation}
 Thus, we
have $$ \|P_u - P_v\| < 1 - \frac{\beta}{3}e^{-Lr} $$ and the
result follows since $Lr \le 1$. \qed

\Aproof{Lemma}{InitialDraw}  We calculate the probability $p$ of making a proper move. We will use the notation defined in
the proof of Lemma \ref{GeoProb}. Let $u$ be an arbitrary point in
$K$. We have that
$$\begin{array}{rcl}
p_u & = & \int_K \min\left\{ \frac{f(x)}{f(u)}, \ 1\right\} q(x|u)dx
\geq \beta e^{-Lr} \int_{B_u \cap K} q(x|u)dx  \geq \beta e^{-Lr} \frac{1}{3}, \
\end{array}
$$
where we used  that
$\inf_{x\in B(y,r)} f(x)/f(y) \geq \beta e^{-Lr}$ by Lemma \ref{Lemma1} and Lemma \ref{GeoLips} and the bound (\ref{bound over H1 and H2}) for the case that
$u = v$ so that $B_u = H_1 \cup H_2$. Since $Lr < 1$
we conclude that $p_u \geq \beta/3e $.

We then note that for $Q(A)>0$ the ratio $Q_0(A)/Q(A)$
is  bounded above by   $\sup_{x \in K} d Q_0(x)/dQ(x)$;
$d Q_0(x)/dx$ is bounded above by  $p_u^{-1} e^{-\|x\|^2/2\sigma^2} \cdot (2\pi \sigma^2)^{-d/2} \leq p_u^{-1} \cdot (2\pi \sigma^2)^{-d/2}$; and $d Q(x)/dx$ is bounded over $x \in K$ below by  $(2\pi)^{-d/2}$ $\det(J^{1/2})$ $e^{-\frac{1}{2}x'Jx}$ $\beta^{1/2} \geq (2\pi)^{-d/2} \lambda_{min}^{d/2} e^{-\frac{1}{2}\|K\|^2_J} \beta^{1/2}$, where $\beta=e^{-2(\epsilon_1 +\epsilon_2\|K\|_J^2/2)}$. Thus, we can bound
$$
 \begin{array}{rcl}
\max_{A \in \mathcal{A}: Q(A)>0} \frac{Q_0(A)}{Q(A)}
& & \leq
p_u^{-1} \sigma^{-d }
 \lambda_{min}^{-d/2} e^{\frac{1}{2}\|K\|^2_J} \beta^{-1/2} \\
& & \leq  3e  [120\sqrt{d}\lambda_{max} \|K\|/\sqrt{\lambda_{min}}]^d
e^{\frac{1}{2}\|K\|^2_J} \beta^{-3/2} \\
&  & \leq  3 [120\|K\|_J^2]^d
e^{3\epsilon_1 +
2\epsilon_2\|K\|^2_J +1},
\end{array}
$$
where we used the bound on $\sigma$ given in (\ref{Eq:Bound:sigma}), and the fact that $\|K\|_J \geq \sqrt{\lambda_{\min}} \|K\|$ and
$\|K\|_{J} > \sqrt{d} \ \sqrt{\lambda_{max}/\lambda_{min}}$ (cf. Comment \ref{comment: on conditions}).

The remaining results in the Lemma follow by invoking the CLT conditions. \qed

 \Aproof{Theorem}{NNN}  We have that, for $\lambda^B$ denoting the random variable with law
$Q_B$ and $\lambda$ denoting the random variable with law $Q$,  and $MSE( \est |X)$ denoting the mean square error $E[(\widehat{\mu}_{g} - \mu_g)^2 |X]$  conditional on the element
$\lambda^{0,B}$ drawn according to $X= \lambda^B$ or $X=\lambda$:
$$
\begin{array}{rcl}
MSE( \est ) &  = &  \displaystyle E_{Q_B} \left[ MSE( \est |\lambda^B) \right] = \displaystyle E_Q \left[ MSE( \est |\lambda) \frac{dQ_B(\lambda)}{dQ(\lambda)} \right]  \\
& = & \displaystyle E_Q \left[ MSE( \est|\lambda ) \right] + E_Q \left[ MSE( \est |\lambda) \left(\frac{dQ_B(\lambda)}{dQ(\lambda)}-1\right) \right] \\
& \leq & \displaystyle E_Q \left[ MSE( \est |\lambda ) \right] +   4  \bar g^2E_Q \[\left| \frac{dQ_B(\lambda)}{dQ(\lambda)}-1\right|\]  \\
& = & \displaystyle (\sigma^2_{g,N}/N) +  8\bar g^2 \| Q_B - Q\|_{TV}, \\
\end{array}
$$
\noindent where $\sigma^2_{g,N}$ is $N$ times the variance of the sample
average when the Markov chain starts from the stationary distribution $Q$.
We also used the fact that $\|Q_B-Q\|_{TV} = \frac{1}{2} \int | dQ_B/dx- dQ/dx| dx$.

The bound on $\sigma^2_{g,N}$ will depend on the particular
scheme, as discussed below.  We begin by bounding the burn-in period $B$.

We require that the second term in the bound for $MSE( \est )$ to be smaller than $\varepsilon/3$,
which is equivalent to imposing that $ \| Q_B - Q\|_{TV} <
\frac{\varepsilon}{24 \bar g^2}$. Using the conductance theorem of  \cite{LS} restated in equation
(\ref{LS}), since
$Q_0$ is  $M$-warm with respect to $Q$, we require that
$$
\begin{array}{rcl}
\sqrt{M}\left( 1 - \frac{\phi^2}{2} \right)^B & \leq & \displaystyle \sqrt{M} e^{-B\frac{\phi^2}{2}} \leq \frac{\varepsilon}{24 \bar g^2}\text{ or }
B \geq   \frac{2}{\phi^2} \ln \left(
\frac{24\sqrt{M} \bar g^2}{\varepsilon} \right).
\end{array}
$$

Next we bound  $\sigma^2_{g,N}$.  Specifically, we determine the
number of post-burn-in iterations $N_{lr}$, $N_{ss}$, or $N_{ms}$
needed to set $MSE( \est ) \leq \varepsilon$.

 1. To bound $N_{lr}$,  note that $\sigma^2_{g,N} \leq \gamma_0 \frac{4}{\phi^2}$ where the last
inequality follows from the conductance-based covariance bound of \cite{LS} restated in equation (\ref{Cor:Cov}). Thus, $N_{lr} =
\frac{\gamma_0}{\varepsilon} \frac{6}{\phi^2}$ and $B$ set above
suffice to obtain $MSE(\est) \leq \varepsilon$.

 2. To bound $N_{ss}$, we first must choose a spacing $S$ to ensure
that the autocovariances are sufficiently small. We
start by bounding
$$\sigma^2_{g,N} \leq \gamma_0 + 2 N |\gamma_S| \leq \gamma_0 + 2 N  \gamma_0 \left( 1 - \frac{\phi^2}{2} \right)^S, $$
\noindent where we used the conductance-based covariance bound of \cite{LS} restated in equation (\ref{Cor:Cov}) and that
$\lambda^{i,B}$ and $\lambda^{i+1,B}$ are spaced by $S$ steps of
the chain. By choosing $S$ as
$$
\displaystyle  \left (1 - \frac{\phi^2}{2} \right)^S \leq
e^{-S\frac{\phi^2}{2}}  \leq \displaystyle \frac{\varepsilon}{6
\gamma_0}, \text{ or } S \geq \displaystyle \frac{2}{\phi^2}
\ln\left( \frac{6 \gamma_0}{\varepsilon} \right),
$$
and using $\displaystyle N_{ss} = \frac{3\gamma_0}{\varepsilon}$,
we obtain
$$
\begin{array}{rcl}
MSE(\hat \mu_{g}) &  \leq &\displaystyle  \frac{1}{N_{ss}} \left( \gamma_0 + 2 N_{ss}  |\gamma_S|\right) +  8\bar g^2 \| Q_B - Q \|_{TV} \\
& \leq & \displaystyle\frac{\varepsilon}{3\gamma_0} \left( \gamma_0 + 2 \frac{3\gamma_0}{\varepsilon}  \gamma_0 \frac{\varepsilon}{6 \gamma_0} \right) +  \bar g^2 \frac{\varepsilon}{3 \bar g^2} \leq \varepsilon
\end{array}
$$

3. To bound $N_{ms}$, we observe, using that $\lambda^{i,B}, i=1,2,...,$ are i.i.d. across $i$, that $MSE(\est) \leq \frac{\gamma_0}{N_{ms}} + \varepsilon/3 \leq \varepsilon$ provided that $N_{ms} \geq  2
\gamma_0/(3 \varepsilon)$. $\qed$ \\

\Aproof{Theorem}{Cond:Exp} Given $$ K = B ( 0, \|K\|) \ \text{
where } \ \|K\|^2 = c d ,
$$ condition C.1 holds by an argument given in proof of Ghosal's Lemma 4.  Let $\lambda_n(c) =  \sqrt{\frac{c d }{n}} B_{1n}(0) +  \frac{c d}{n}
B_{2n}(c)$. Our condition C.2 is satisfied by an argument similar to that given in the proof of
Ghosal's Lemma 1 with $$\epsilon_1=O \( \lambda_n(c) \|s\|^2  \)=O_p(\lambda_n(c)d) = O_p(d^{3/2}/n^{1/2}) = o_p(1) \ \ \mbox{and}$$ $$ \epsilon_2 = O \( \lambda_n(c) \) = O_p\( d^{1/2}/n^{1/2} \) = o_p(1/d),
$$
and our condition C.3 is satisfied since $\epsilon_2
\|K\|^2_J = o_p(1)$. \qed

\begin{remark}
 Ghosal \cite{G2000} proves his results for the support set $K' = B(0, C
\sqrt{d} \log d)$. His arguments actually go through for the
support set $K = B(0, C \sqrt{d} )$  due to the concentration of
normal measure under $d \to \infty$ asymptotics. For details, see
\cite{BC}.
\end{remark}

\Aproof{Theorem}{Cond:CExp} Take $ K = B ( 0, \|K\|)$, where $
\|K\|^2 = C d_1 $ for some $C$ sufficiently large independent of
$d$ (see \cite{BC} for details). Let $\lambda_n(c) =  \sqrt{\frac{c d }{n}} B_{1n}(0) +  \frac{c d}{n}
B_{2n}(c)$. Then condition C.1 is satisfied
by the argument given in the proof of Ghosal's Lemma 4 and NE.3.
Further, condition C.2 is satisfied by the argument similar to that given in the
proof of Ghosal's Lemma 1 and by NE.3 with \bs \epsilon_1 &=
O_p\left(\delta_{1n}d^{1/2} + \delta_{2n}d  + \lambda_n(C)(\delta_{1n}d^{1/2} + \delta_{2n}d^{1/2} +  d)\right) = o_p(1), \\
\epsilon_2 &= O_p \( \lambda_n(C) \) = o_p(d^{1/2}/n^{1/2})=o_p(1/d),
\end{split}\end{align}
and condition C.3 is satisfied since $\epsilon_2 \|K\|^2_J =o_p(1).$ \qed

\begin{remark}
For further details, see \cite{BC}.
\end{remark}


\Aproof{Theorem}{Theorem:Zest} We will first establish the
following linear approximation for $S_n(\theta)$ in a neighborhood
of $\theta_0$
\begin{equation}\label{HL-Zest-2}
\sup_{\|\theta - \theta_0\| \leq C \sqrt{d/n} } \| S_n(\theta) -
S_n(\theta_0) - n^{1/2}A ( \theta - \theta_0) \| = o_p\(d^{-1/2}\)
\end{equation}
for any fixed constant $C>0$. For notational convenience let
{\small
\begin{equation}\label{Def:DeltaW} \delta_n (\theta) = S_n(\theta)
- S_n(\theta_0) - n^{1/2}A ( \theta - \theta_0), \  W_n(\theta) =
S_n(\theta) - S_n( \theta_0) - \Ep\[ S_n(\theta) - S_n(\theta_0)\].
\end{equation}}
Let  $\mathcal{F}_n= \{\eta' (m(X,\theta)- m(X,\theta_0)) \ : \|\theta - \theta_0\| \leq \rho_n, \eta \in S^{d_1}  \}$.
Under condition ZE.1, we apply the following maximal inequality
adopted from He and Shao \cite{HS00}  (see \cite{BC-qrID} for details) to an empirical process indexed by members of
$\mathcal{F}_n$:
{\small \begin{equation}\label{Eq:MaxIneq}
\sup_{f \in \mathcal{F}_n} | n^{-1/2} \sum_{i=1}^n (f(X_i)-E[f(X_i)] )| = O_p\( \sqrt{V\log n} \(
\sup_{f\in \mathcal{F}_n} E[f^2]  +  n^{-1} V M^2 \log
n\)^{1/2} \). \end{equation}}
Here the multiplier $ \sqrt{V}$ arises as the order of the uniform bracketing entropy integral,
where $V$ is the VC dimension of  a VC function class $\mathcal{F}_n$ or an entropically equivalent class $\mathcal{F}_n$.
We assumed in ZE.1 that $V = O(d)$. Also  $M$ is the a.s. bound on the envelope of $\mathcal{F}_n$, assumed to be of order $O(\sqrt{d})$.
Finally,  we assumed that $\sup_{f\in
\mathcal{F}_n} (E[f^2])^{1/2} = O(\rho_n^{\alpha})$. Therefore, we have that
 uniformly in $\theta\in \Theta_n$
\begin{equation}\label{Eq:Approx}
\begin{array}{rcl}\|W_n(\theta)\| & = &  O_p\left( \
\sqrt{d \log n} \(  \| \theta - \theta_0\|^{2\alpha}   + n^{-1} d
M^2 \log n \)^{1/2}  \right)\\
&= & O_p\left( \sqrt{d\log n} \|\theta-\theta_0\|^{\alpha} +
n^{-1/2}
d^{3/2} \log n \).\\
\end{array}
\end{equation}

Note that (\ref{Eq:Approx}) and an expansion with an integral reminder around $\theta-\theta_0$ shows that uniformly in $\theta \in \Theta_n$
$$
\begin{array}{rcl}\| \delta_n(\theta) \| & \leq & \|W_n(\theta)\| + \|
\nabla^2 E[S_n(\xi)] \cdot [\theta - \theta_0, \theta-\theta_0] \|\\
& = & O_p\left( \ d^{1/2}\log^{1/2} n \|\theta-\theta_0\|^{\alpha} + n^{-1/2} d^{3/2} \log n \right) + \\
& + &  O_p\( \sqrt{dn} \|\theta -
\theta_0\|^2\)
\end{array}
$$ where $\xi$ lies between $\theta$ and $\theta_0$ and we used ZE.2 that
imposes
$\|\nabla^2E[S_n(\xi)] \cdot [\gamma, \gamma] \| = O(\sqrt{dn}\|\gamma\|^2)$.
The condition (\ref{HL-Zest-2}) follows from the growth condition
ZE.3(a).

Building upon (\ref{HL-Zest-2}), Lemmas \ref{Lemma:Zest} and
\ref{Lemma:Kc2} verify that conditions C.1-C.3 hold proving
Theorem \ref{Theorem:Zest}. $\qed$

\begin{lemma}\label{Lemma:Zest}
Under conditions ZE.1-ZE.3, conditions C.2 and C.3 hold  for $K=B(0,C\sqrt{d})$ for any fixed constant $C > 0$.
\end{lemma}
\Aproof{Lemma}{Lemma:Zest} Let $s = -(A'A)^{-1}A'S_n(\theta_0)$ be
a first order approximation for the extremum estimator. For
$\theta = \theta_0 + (s+\lambda)/\sqrt{n}$ and $\tilde \theta = \theta_0 + s/\sqrt{n}$
$$
\begin{array}{rcl}
\ln \ell (\lambda) & = &
-\|S_n(\theta)\|^2 +
\|S_n(\tilde \theta)\|^2\\
& = & -\lambda'A'A\lambda - \| r_n\|^2 -
2r_n'A\lambda - 2
r_n'S_n(\tilde \theta)\\
& = & -\lambda'A'A\lambda  + o_p(1),
\end{array}$$ where $r_n = \delta_n(\theta) - \delta_n(\tilde \theta)$ for  $\delta_n(\theta)$ defined in (\ref{Def:DeltaW}). Indeed, using
(\ref{HL-Zest-2}) we have $\|\delta_n(\theta)\|=o_p(d^{-1/2})$ and $\|\delta_n(\tilde \theta)\|=o_p(d^{-1/2})$ uniformly over $\lambda \in K$; using (\ref{HL-Zest-2}) we have $\|S_n(\tilde \theta)\|=O_p(d^{1/2})$; and moreover,
$\|\lambda\|=O(d^{1/2})$, and $\|s\| = O_p(d^{1/2})$ by
Chebyshev inequality. Thus, conditions C.2 and C.3 follow with $\epsilon_1 =
o_p(1)$, $\epsilon_2 = 0$, and $J=2A'A$. $\qed$


\begin{lemma}\label{Lemma:Kc2}
Under the conditions ZE.1, ZE.2, and ZE.3 there exist a
constant $C>0$ such that by setting $K = B(0,C\sqrt{d})$ we have
$\int_{K^c} \ell(\lambda) d\lambda = o_p\( \int_{K} \ell(\lambda)
d\lambda \)$ and condition C.1 holds. \end{lemma}
\Aproof{Lemma}{Lemma:Kc2} For notational convenience we conduct the
proof in the original parameter space. Let $\tilde\theta = \theta_0 + s/\sqrt{n}$ and $\varepsilon>0$ be any small positive constant. Since $\|s\| = O_p(d^{1/2})$, there is a constant $\hat C$ such that
$\|s\|\leq \hat C d^{1/2}$, with asymptotic probability no smaller
than $1-\varepsilon$. Below we replace the last phrase by ``\wpe".

Now, since $\Ep[S_n(\theta_0)] =
0$, we have that
\begin{equation}\label{Eq:BoundSn01}
S_n(\theta) = W_n(\theta) + S_n(\theta_0) + \Ep[ S_n(\theta)],
\end{equation} where $W_n(\theta)$ is defined in
(\ref{Eq:Approx}).

Next, define for $C \geq \hat C + \tilde C$ the sets \begin{equation}\label{Def:KK}\widetilde K =
B\(\theta_0,\ \tilde C \sqrt{d/n}\) \subseteq \hat K =
B\(\tilde \theta, C \sqrt{d/n}\),\end{equation}
where the inclusion holds \wpe.
 Note that these sets
are centered on different points. We will show that for a
sufficiently large constant $\tilde C $
$$ \int_{\hat K^c} \exp(-\|S_n(\theta)\|^2 ) d\theta = o_p\( \int_{\hat K} \exp(-\|S_n(\theta)\|^2 ) d\theta
\), $$
which implies the claim of the lemma.

{\em Step 1. Relative bound on $\|S_n(\theta_0)\|$.}  Note that
$\|S_n(\theta_0)\|=O_p(d^{1/2})$ by Chebyshev inequality. Using
equation (\ref{Eq:ZE2}) of condition ZE.2,  we have that
$$ \|\Ep[S_n(\theta)]\|^2 \geq \left( \sqrt{n} ( \sqrt{\mu}
\|\theta-\theta_0\| \wedge \delta )\right)^2  \geq \left( \
\tilde C\sqrt{\mu}\sqrt{d} \ \right)^2, \ \ \forall \theta \in \tilde K^c$$
since $\|\theta -\theta_0\| \geq \tilde C \sqrt{d/n}$. Therefore, there exists $\tilde C$ such that \wpe
\begin{equation}\label{Eq:BoundSn02}  \| \Ep[S_n(\theta)] \| > 5
\|S_n(\theta_0)\| \text{ uniformly in } \theta \in \widetilde
K^c.\end{equation}

{\em Step 2. Relative bound on $\|W_n(\theta)\|$.} Using equation
(\ref{Eq:Approx}), we have that for uniformly in $\theta \in \Theta_n
\subset B(0,T_n)$
$$
\left\| W_n(\theta) \right\| = O_p\left( \ \sqrt{d \log n}
\|\theta - \theta_0\|^{\alpha} + n^{-1/2}d^{3/2}\log n\right),$$
Building on that,  we will show
that $\|W_n(\theta)\|$ $=$ $o_p\(\sqrt{n} ( \delta \wedge \|\theta
- \theta_0\| ) \)$ uniformly on  $\theta \in \widetilde K^c$, and
therefore
\begin{equation}\label{Eq:BoundSn03}
\|W_n(\theta)\| = o_p(\| E [ S_n(\theta) ] \| ), \ \mbox{uniformly
in} \  \theta \in \tilde K^c.
\end{equation}

For the case that $\delta \leq \|\theta - \theta_0\| \leq T_n$ it
suffices to have $\ \sqrt{d \log n}
T_n^{\alpha} + n^{-1/2}d^{3/2}\log n = o (n^{1/2})$. On the other hand, for $C\sqrt{d/n} \leq \|\theta - \theta_0\|\leq \delta$ it suffices to
have $\sqrt{d \log n}
\|\theta - \theta_0\|^{\alpha} + n^{-1/2}d^{3/2}\log n = o( \sqrt{n} \|\theta - \theta_0\|)$.  Indeed, $\sqrt{d \log n}
\|\theta - \theta_0\|^{\alpha} =( \sqrt{n} \|\theta - \theta_0\|)$ if
$\sqrt{d \log n} = o(\sqrt{n} \|\theta - \theta_0\|^{1-\alpha})$,
which is implied by $ \sqrt{d \log n} =o(\sqrt{n}(d/n)^{\frac{1-\alpha}{2}})$.  Moreover,
$n^{-1/2}d^{3/2}\log n = o( \sqrt{n} \|\theta - \theta_0\|)$ if $n^{-1/2}d^{3/2}\log n = o( \sqrt{n} \sqrt{d/n})$. All of the above conditions hold under condition ZE.3.

{\em Step 3. Lower bound on $\|S_n(\theta)\|$.} We will show that
\begin{equation}\label{Bound:Q} \|S_n(\theta)\|^2  = \| \Ep[S_n(\theta)]
+ S_n(\theta_0) + W_n(\theta) \|^2 \geq \frac{1}{2}
 \| \Ep[S_n(\theta)]\|^2\end{equation}
 uniformly for all $\theta \in \widetilde K^c$ wp $1- 2\varepsilon$.

For any two vectors $a$ and $b$, we have $\| a + b \|^2 \geq
(\|a\| - \|b\|)^2 = \|a\|^2 - 2\|a\|\|b\| + \|b\|^2 \geq
\|a\|^2\(1 - 2\|b\|/\|a\| \)$. Applying this relation with
$a=E[S_n(\theta)]$ and $b = W_n(\theta)+S_n(\theta_0)$,(\ref{Eq:BoundSn02}), and
(\ref{Eq:BoundSn03}), we obtain (\ref{Bound:Q}).

{\em Step 4. Bounding the integrals.} Using (\ref{Bound:Q}) and
 ZE.2  wp $1- 3\varepsilon$
$$\begin{array}{rcl} &\int_{\hat K^c}&  \exp  (-\|S_n(\theta)\|^2)d\theta
\leq  \int_{\widetilde K^c} \exp  (-\|S_n(\theta)\|^2)d\theta \\
& &\leq  \int_{\widetilde K^c}  \exp  ( -
\frac{1}{2}\|E\[S_n(\theta)\]\|^2) d\theta \\
& & \leq   \int_{\widetilde K^c} \exp( - \frac{1}{2}\mu n \| \theta - \theta_0\|^2 )
d\theta + \int_{\widetilde K^c} \exp( - \frac{1}{2}\mu n \delta^2 ) d\theta \\
& & \leq  (2\pi)^{d/2}\(n\mu\)^{-\frac{d}{2}} P( \|U\| > \tilde C
\sqrt{d/n}
) +   \exp( - \frac{1}{2}\mu n \delta^2 ) \vol( \Theta_n )\\
 & & \leq
(2\pi)^{d/2}\(n\mu\)^{-\frac{d}{2}} \exp\( -
\frac{(\tilde C-1/\sqrt{\mu})^2\mu}{2}d \) + \nu_d
T_n^d \exp( - \frac{1}{2} \mu n \delta^2) \\
\end{array}
$$
where $\nu_d$ is the volume of the $d$-dimensional unit ball, which goes to
zero as $d$ grows, and $U\sim N(0, \frac{1}{\mu n} \ I_d)$. In the first line
we used the inclusion (\ref{Def:KK}), and  in the last line we used a standard Gaussian concentration inequality, Proposition 2.2 in Talagrand \cite{Talagrand1994}, and the fact
that $\Ep[\|U\|] \leq (\Ep[\|U\|^2])^{1/2} = \frac{1}{\sqrt{\mu}}
\sqrt{d/n}$.

On the other hand, by Lemma \ref{Lemma:Zest} we have
$$-\|S_n(\theta)\|^2 + \|S_n(\tilde\theta)\|^2  = n\|A(\theta - \tilde\theta)\|^2 + o_p(1)$$
uniformly for $\theta \in \hat K$. This yields that wp $1- \varepsilon$
$$\begin{array}{rcl} \int_{\hat K} \exp( - \|S_n(\theta)\|^2 )  & d\theta &
\geq  \exp( -\|S_n(\tilde\theta)\|^2 ) \int_{\hat K} \exp( - n\|A(\theta -
\tilde\theta)\|^2 + o_p(1) ) d\theta \\
& \geq &  \exp(- C_2 d) \int_{\hat K} \exp( - C_1n\|\theta -
\tilde\theta\|^2 ) d\theta \\
& \geq &  \exp(- C_2   d)
(2\pi)^{\frac{d}{2}} (C_1n)^{-\frac{d}{2}} ( 1- P( \|U\| \leq C \sqrt{d/n}) ) \\
& \geq & \exp(-C_2  d ) (2\pi)^{\frac{d}{2}} (C_1n)^{-\frac{d}{2}} (1- o(1)) \\
\end{array}
$$ where constant $C_1$ is maximal eigenvalue of $A'A$, constant $C_2$ is such that $\|S_n(\tilde\theta)\|^2 \leq  C_2 d$ wp $1-\varepsilon$ by Lemma \ref{Lemma:Zest}, $U \sim N(0,\frac{1}{C_1n}I_d)$. In the last line we used the standard
Gaussian concentration inequality, Proposition 2.2 in Talagrand
\cite{Talagrand1994}, with constant
$C > 2/\sqrt{C_1}$ to get $P( \|U\| \leq
C \sqrt{d/n})  = o(1)$.

Finally, we obtain that wp $1- 5\varepsilon$
{\small $$ \frac{\int_{\hat K^c} \exp( -\|S_n(\theta)\|^2 )
d\theta}{\int_{\hat K} \exp( -\|S_n(\theta)\|^2 ) d\theta} \leq
\frac{(2\pi)^{\frac{d}{2}}(\mu n)^{-\frac{d}{2}} \exp\( - \frac{(\tilde C
-1/\sqrt{\mu})^2\mu}{2}d \) + \nu_d T_n^d \exp( - \frac{1}{2} \mu
n \delta^2)}{\exp(-C_2 d) \ (2\pi)^{d/2} (C_1n)^{-d/2}(1+o(1))}
$$}
\!\!where the right hand side is $o(1)$
by choosing  $\tilde C > 0$ sufficiently large, and noting that terms $(2\pi)^{d/2} n^{-d/2}$ cancel
and that $d\ln T_n = o(n)$ by condition ZE.3.

Since $\varepsilon >0$ can be set as small as we like, the conclusion follows. \qed


\section{Bounding log-$\beta$-concave
functions}\label{App:BetaLog}

\begin{lemma}\label{STRbr}
Let $f:\RR \to \RR $ be a unidimensional log-$\beta$-concave
function. Then there exists a logconcave function $g:\RR \to \RR$
such that
$$\beta g(x) \leq f(x) \leq g(x) \ \ \mbox{for every} \ \ x \in \RR.$$
\end{lemma}
\begin{proof}
Consider $h(x) = \ln f(x) $ a $(\ln \beta)$-concave function. Now,
let $m$ be the smallest concave function greater than $h(x)$ for
every $x$, that is,  $$\begin{array}{rcl}
m(x) & = & \displaystyle \sup \left\{ \sum_{i=1}^k \lambda_i h(y_i) : k\in \NN, \lambda \in \RR^k, \lambda \geq 0, \sum_{i=1}^k\lambda_i = 1, \sum_{i=1}^k \lambda_i y_i = x \right\}. \\
\end{array}$$

Recall that the epigraph of a function $w$ is defined as $epi_w =
\{ (x,t) : t \leq w(x)\}$. Using our definitions, we have that
$epi_m = \conv(epi_h)$ (the convex hull of $epi_h$), where both
sets lie in $\RR^2$. In fact, the values of $m$ are defined only
by points in the boundary of $\conv(epi_h)$. Consider $(x,m(x))
\in epi_m$, since the epigraph is convex and this point is on the
boundary, there exists a supporting hyperplane $H$ at $(x,m(x))$.
Moreover, $(x,m(x)) \in \conv(epi_h\cap H)$. Since $H$ is one
dimensional, $(x,m(x))$ can be written as convex combination of at
most $2$ points of $epi_h$.

Furthermore, by definition of log-$\beta$-concavity, we have that
$$ \ln 1 / \beta \geq \sup_{\lambda \in [0,1], y, z}  \lambda h(y) + (1-\lambda) h(z) -  h\left( \lambda y + (1-\lambda)z  \right). $$
\noindent Thus, $ h(x) \leq m(x) \leq h(x) + \ln (1/\beta)$.
Exponentiating gives $ f(x) \leq g(x) \leq \frac{1}{\beta} f(x),$
where $g(x)= e^{m(x)}$ is a logconcave function.
\end{proof}

\end{appendix}



\end{document}